\documentclass[11pt]{amsart}
\usepackage{changes}
\definechangesauthor[name=JC, color=blue]{JC}
\definechangesauthor[name=YB, color=orange]{YB}

\usepackage[utf8]{inputenc} 	
\usepackage[T1]{fontenc}
\usepackage{lmodern}
\usepackage{yhmath}

\usepackage{amsthm,amsmath,amssymb,amsbsy,bbm,mathrsfs,supertabular,
eurosym,graphicx,enumitem,xcolor}
\usepackage{mathtools}

\newtheoremstyle{BBstyle0}  {}{}{\itshape}{}{\bfseries}{}{6pt}{}
\newtheoremstyle{BBstyle1}  {3pt}{3pt}{\rmfamily}{}{\itshape}{: }{3pt}{}
\newtheoremstyle{BBstyle2}  {3pt}{3pt}{\itshape}{}{\bfseries\large}{}{0pt}{}
\newtheoremstyle{BBstyle3}  {}{}{\itshape}{}{\bfseries}{: }{3pt}{}
\newtheoremstyle{BBstyle4}  {}{}{\rmfamily}{}{\bfseries}{}{6pt}{}
  
\usepackage[normalem]{ulem} 
\usepackage[authoryear]{natbib}
\newtheorem{thm}{Theorem}
\newtheorem{lem}{Lemma}
\newtheorem{prop}{Proposition}

\newtheorem{cor}{Corollary}
\newtheorem{ass}{Assumption}

\theoremstyle{definition}
\newtheorem{exa}{Example}

\usepackage[english]{babel}

\usepackage{hyperref}  		     

\newcommand{\pa}[1]{\left({#1}\right)}
\newcommand{\norm}[1]{\left\|{#1}\right\|}
\newcommand{\cro}[1]{\left[{#1}\right]}
\newcommand{\ab}[1]{\left|{#1}\right|}
\newcommand{\ac}[1]{\left\{{#1}\right\}}

\newcommand{\argmax}{\mathop{\rm argmax}}

\newcommand{\Card}{\mathop{\rm Card}\nolimits}

\newcommand{\dfleche}[1]{\,\displaystyle{\mathop{\longrightarrow}_{#1}}\,}

\newcommand{\CV}[1]{\dfleche{#1}}

\newcommand{\E}{{\mathbb{E}}}
 
\renewcommand{\L}{{\mathbb{L}}}
\newcommand{\N}{{\mathbb{N}}}
\renewcommand{\P}{{\mathbb{P}}}
 
\newcommand{\R}{{\mathbb{R}}}

\newcommand{\sB}{{\mathscr{B}}}
\newcommand{\sC}{{\mathscr{C}}}

\newcommand{\sE}{{\mathscr{E}}}
\newcommand{\sF}{{\mathscr{F}}}
\newcommand{\sG}{{\mathscr{G}}}

\newcommand{\sL}{{\mathscr{L}}} 
\newcommand{\sM}{{\mathscr{M}}}

\newcommand{\sP}{{\mathscr{P}}}
\newcommand{\sQ}{{\mathscr{Q}}} 

\newcommand{\sT}{{\mathscr{T}}}

\newcommand{\sW}{{\mathscr{W}}}
\newcommand{\sX}{{\mathscr{X}}}
\newcommand{\sY}{{\mathscr{Y}}} 
\newcommand{\sZ}{{\mathscr{Z}}}

\DeclareMathAlphabet{\mathscrbf}{OMS}{mdugm}{b}{n}

\newcommand{\sbP}{{\mathscrbf{P}}}
\newcommand{\sbQ}{{\mathscrbf{Q}}}

\newcommand{\cB}{{\mathcal{B}}}
\newcommand{\cC}{{\mathcal{C}}}

\newcommand{\cE}{{\mathcal{E}}}
\newcommand{\cF}{{\mathcal{F}}}
 
\newcommand{\cH}{{\mathcal{H}}}

\newcommand{\cN}{{\mathcal{N}}}

\newcommand{\cP}{{\mathcal{P}}}
 
\newcommand{\cR}{{\mathcal{R}}}
 
\newcommand{\cT}{{\mathcal{T}}}
\newcommand{\cU}{{\mathcal{U}}}
\newcommand{\cV}{{\mathcal{V}}}
\newcommand{\cW}{{\mathcal{W}}}
\newcommand{\cX}{{\mathcal{X}}}
\newcommand{\cY}{{\mathcal{Y}}}

\newcommand{\cbP}{\boldsymbol{\mathcal{P}}}

\newcommand{\gh}{{\mathbf{h}}}

\newcommand{\gp}{{\mathbf{p}}}

\newcommand{\gx}{{\mathbf{x}}}

\newcommand{\gA}{{\mathbf{A}}}

\newcommand{\gP}{{\mathbf{P}}}
\newcommand{\gQ}{{\mathbf{Q}}} 
\newcommand{\gR}{{\mathbf{R}}}
 
\newcommand{\gT}{{\mathbf{T}}}

\newcommand{\bs}[1]{\boldsymbol{#1}}

\newcommand{\bsG}{{\bs{G}}}

\newcommand{\bsS}{{\bs{S}}} 
\newcommand{\bsT}{{\bs{T}}}

\newcommand{\bsX}{{\bs{X}}}

\newcommand{\geta}{\bs{\eta}}
\newcommand{\ggamma}{\bs{\gamma}}
\newcommand{\gGamma}{\bs{\Gamma}}

\newcommand{\gmu}{{\bs{\mu}}}

\newcommand{\gtheta}{{\bs{\theta}}}

\newcommand{\gup}{\bs{\upsilon}}

\newlist{lista}{enumerate}{1}
\setlist[lista,1]{label=\alph*),ref=\alph*)}

\newlist{listi}{enumerate}{1}
\setlist[listi,1]{label=(\roman*),ref=(\roman*),align=left}

\newcommand{\eref}[1]{(\ref{#1})}
\newcommand{\eqd}{\stackrel{\mathrm{def}}{=}}
\renewcommand{\ge}{\geqslant}
\renewcommand{\le}{\leqslant}
\newcommand{\1}{1\hskip-2.6pt{\rm l}}

\newcommand{\<}{{\langle}}
\renewcommand{\>}{{\rangle}}

\newcommand{\etc}[1]{#1_1,\ldots,#1_n}
\newcommand{\Etc}[2]{#1_1,\ldots,#1_{#2}}

\newcommand{\st}{\strut}

\newcommand{\on}{^{\otimes n}}

\newcommand{\et}{^{\star}}
\newcommand{\eps}{{\varepsilon}}

\usepackage{booktabs}
\usepackage{footnote}
\usepackage{enumitem}
\usepackage{algorithm}
\usepackage{algorithmic}

\def\bst{\bs{\theta}}
\def\bsT{{\overline{\bs{\Theta}}}}
\def\bsTT{{\bs{\Theta}}}
\def\bsb{{\bs{\beta}}}

\def\bsg{{\ggamma}}
\def\bsG{{\overline{\bs{\Gamma}}}}

\begin{document}
\title[Robust estimation in exponential families]{Robust estimation of a regression function in Exponential Families}
\author{Yannick BARAUD}
\author{Juntong CHEN}
\address{\parbox{\linewidth}{Department of Mathematics (DMATH),\\
University of Luxembourg\\
Maison du nombre\\
6 avenue de la Fonte\\
L-4364 Esch-sur-Alzette\\
Grand Duchy of Luxembourg}}
\email{yannick.baraud@uni.lu}
\email{juntong.chen@uni.lu}

\thanks{This project has received funding from the European Union's Horizon 2020 research and innovation programme under grant agreement N\textsuperscript{o} 811017}
\date{\today}
\begin{abstract}
We observe $n$ pairs of independent {(but not necessarily i.i.d.)} random variables $X_{1}=(W_{1},Y_{1}),\ldots,X_{n}=(W_{n},Y_{n})$ and tackle the problem of estimating the conditional distributions $Q_{i}\et(w_{i})$ of $Y_{i}$ given $W_{i}=w_{i}$ for all $i\in\{1,\ldots,n\}$. Even though these might not be true, we base our estimator on the assumptions that the data are i.i.d.\ and the conditional distributions of $Y_{i}$ given $W_{i}=w_{i}$ belong to a one parameter exponential family $\overline \sQ$ with parameter space given by an interval $I$. More precisely, we pretend that these conditional distributions take the form  $Q_{\bst(w_{i})}\in \overline \sQ$ for some $\bst$ that belongs to a VC-class $\bsT$ of functions  with values in $I$. For each $i\in\{1,\ldots,n\}$, we estimate $Q_{i}\et(w_{i})$ by a distribution of the same form, i.e.\ $Q_{\widehat\bst(w_{i})}\in \overline \sQ$, where  $\widehat \bst=\widehat \bst(X_{1},\ldots,X_{n})$ is a well-chosen estimator with values in $\bsT$. We establish non-asymptotic exponential inequalities for the upper deviations of a Hellinger-type distance between the true conditional distributions of the data and the estimated one based on the exponential family $\overline \sQ$ and the class of functions $\bsT$. We show that our estimation strategy is robust to model misspecification, contamination and the presence of outliers. Besides, when the data are truly i.i.d., the exponential family $\overline \sQ$ suitably parametrized and the conditional distributions $Q_{i}\et(w_{i})$ of the form $Q_{\bst\et(w_{i})}\in\overline \sQ$ for some unknown H\"olderian function $\bst\et$ with values in $I$, we prove that the estimator $\widehat \bst$ of $\bst\et$ is minimax  (up to a logarithmic factor). Finally, we provide an algorithm for calculating $\widehat \gtheta$ when $\bsT$ is a VC-class of functions of low or {moderate} dimension and we carry out a simulation study to compare the performance of $\widehat \gtheta$ to that of the MLE and median-based estimators. The proof of our main result relies on an upper bound, with explicit numerical constants, on the expectation of the supremum of an empirical process over a VC-subgraph class. This bound can be of independent interest. 
\end{abstract}

\maketitle

\section{Introduction}
In order to motivate the statistical problem we wish to solve here, let us start with a preliminary example.
\begin{exa}[Logit regression]\label{exa-C}
We study a cohort of $n$ patients with respective clinical characteristics $W_{1},\ldots,W_{n}$ with values in $\R^{d}$. For the sake of simplicity we shall assume that $d$ is small compared to $n$ even though this situation might not be the practical one.  We associate the label $Y_{i}=1$ to the patient $i$ if {she/he} develops the disease $D$ and $Y_{i}=-1$ otherwise. {The effect of the clinical characteristic $W$ on the probability of developing the disease $D$ is given by the conditional distribution of $Y$ given $W$: $\P\cro{Y=y|W=w}$ which is the quantity we want to estimate. A classical model for it  is the logit one given by
\begin{equation}
\P\cro{Y=y|W=w}=\frac{1}{1+\exp\cro{-y\left<w\et, w\right>}}\in (0,1)\quad \text{for $y\in\{-1,+1\}$},
\label{mod-LRb}
\end{equation}
}where $w\et$ is an unknown vector  and $\<\cdot,\cdot\>$ the  inner product of $\R^{d}$. If we assume that this model is true, the problem {amounts} to estimate $w\et$ on the basis of the observations $(W_{i},Y_{i})$ for $i\in\{1,\ldots,n\}$. 
\end{exa}
A common way of solving this problem is to use the Maximum Likelihood Estimator (MLE for short). In exponential families, the MLE is known to enjoy many nice properties but it also suffers from several defects. First of all, it is not difficult to see that it might not exist. This is in particular the case when a hyperplane separates the two subsets of $\R^{d}$ given by $\cW_{+}=\{W_{i},\, Y_{i}=+1\}$ and $\cW_{-}=\{W_{i},\, Y_{i}=-1\}$, i.e.\  when there exists a unit vector $w_{0}\in\R^{d}$ such that $\<w,w_{0}\>>0$ for all $w\in\cW_{+}$ and  $\<w,w_{0}\><0$ for $w\in\cW_{-}$. In this case, the conditional likelihood function at $\lambda w_{0}$ with $\lambda>0$ can be written as
\[
\prod_{i=1}^{n}\frac{1}{1+\exp\cro{-\lambda Y_{i}\left<w_{0}, W_{i}\right>}}=\prod_{i=1}^{n}\frac{1}{1+\exp\cro{-\lambda \ab{\left<w_{0}, W_{i}\right>}}}\CV{\lambda\to +\infty} 1,
\]
hence the maximal value 1 is not reached. For a thorough study of the existence of the MLE in the logit model we refer to Cand\`es and Sur~\citeyearpar{Cand_s_2020} as well as the references therein. 

Another issue with the use of the MLE lies in the fact that it is not robust and we {shall} illustrate its instability in our simulation study. Robustness is nevertheless an important property in practice since, {going back to Example~\ref{exa-C}}, it may happen that our database contains a few corrupted data that correspond to mislabelled patients  (some patients might have developed a disease which is not $D$ but has similar symptoms) {or that the relation (\ref{mod-LRb}) is only approximately true.} A natural question arises: how can we provide a suitable estimation of {$\P\cro{Y=y|W=w}$} despite the presence of {possibly} corrupted data {or a slight misspecification of the model?}

This is the kind of issue we want to solve here. Our approach is not, however, restricted to the logit model but applies more generally whenever the conditional distribution of $Y$ given $W$ belongs to a one-parameter exponential family. {More precisely, we shall work within the following statistical framework. We observe $n$ pairs of independent, typically non  i.i.d., random variables $X_{1}=(W_{1},Y_{1}),\ldots,X_{n}=(W_{n},Y_{n})$ (with $W_{i}\in\sW$ and $Y_{i}\in\sY$ for $i\in\{1,\ldots,n\}$) and we want to estimate the $n$ conditional distributions $Q_{i}\et(w_{i})$ of $Y_{i}$ when $W_{i}=w_{i}$, $i\in\{1,\ldots,n\}$, without any information about the distributions $P_{W_{i}}$ of the variables $W_{i}$ which are unknown and can be completely arbitrary. In order to do so, we introduce a statistical model for the $Q_{i}\et(w_{i})$. We start from an exponential family $\{Q_{\theta},\, \theta\in I\}$ where $I$ is an interval of $\R$ and consider the family of conditional distributions $\{Q_{\bst(w)}, \bst\in\bsT\}$ where $\bsT$ is a given set of functions from $\sW$ to $I$. This provides a model for the $n$ conditional distributions $Q_{i}\et(w_{i})$ if we pretend that $Q_{i}\et(w_{i})$ takes the form $Q_{\bst\et(w_{i})}$ for some $\bst\et\in\bsT$, i.e.
\begin{equation}
Q_{i}\et(w_{i})=Q_{\theta_{i}\et}\quad\text{with}\quad\theta_{i}\et=\bst\et(w_{i})\quad\text{for }
i\in\{1,\ldots,n\}.
\label{eq-model}
\end{equation}
We shall do as if (\ref{eq-model}) were true although we do not assume it. We merely hope that the set  of conditional distributions $Q_{\theta_{i}\et}$ induced by a suitable element $\bst\et$ of $\bsT$ provides a reasonably good approximation for the true conditional distributions $Q_{i}\et(w_{i})$. Any estimator $\widetilde{\bst}$ of $\bst\et$ leads, by an application of (\ref{eq-model}), to an estimator $Q_{\widetilde{\bst}(w_{i})}$ of the conditional distribution $Q_{i}\et(w_{i})$. 
We measure the risk of such an estimator by a Hellinger-type distance between the conditional distributions $Q_{i}\et(w_{i})$ and their estimators, integrated with respect to the probabilities 
$P_{W_{i}}$ (to be defined in the next section).}

{Given the model indexed by the elements of $\bsT$, instead of estimating $\bst\et$ by the maximum likelihood method as is commonly done, we use for this a $\rho$-estimator $\widehat \gtheta$, whose definition and perfiormance are described in great details in Baraud {\em et al.}~\citeyearpar{MR3595933} and Baraud and Birg\'e ~\citeyearpar{BarBir2018}. The purpose of this replacement is to avoid various drawbacks connected to the use of the MLE:\\
--- It may not exist;\\
--- It is typically difficult to evaluate its performance in a non-asymptotic framework and its analysis generally requires some knowledge or restrictions about the distributions of the $W_{i}$;\\
--- Its performance may be very bad when the model is not exact (misspecification, presence of outliers, contamination, etc.) as demonstrated by our simulations in Section~\ref{sect-ST}.}

{On the contrary a $\rho$-estimator always exists and it enjoys the following properties. \\
\begin{itemize}
\item When the parameter set $\bsT$ is VC-subgraph with VC-dimension $V$, the non-aymptotic risk of $\widehat \bst$ is bounded by the sum of two terms : an approximation term reflecting the distance between the model and the truth and an estimation term corresponding to the risk bound one would get if the model were true. Moreover, this second term only depends on $V$. This risk bound involves explicit constants and  holds under the only assumption that the data $(W_{1},Y_{1}),\ldots,(W_{n},Y_{n})$ are independent;
\item  the estimator $\widehat \bst$ still performs well when the function $\bst\et$ does not belong to $\bsT$ but lies close enough to it; 
\item  the estimator is robust: its performance remains stable when the data set $(W_{1},Y_{1}),\ldots,(W_{n},Y_{n})$ is contaminated or contains outliers or when the statistical model based the exponential family is only approximately correct.
\item  when the model is exact, the exponential family $\{Q_{\theta},\, \theta\in I\}$ is suitably parametrized and $\bsT$ is a H\"olderian class of smoothness, the estimator $\widehat \bst$ is rate optimal (up to a logarithmic factor).
\end{itemize}
}

The work presented here is different from the study of $\rho$-estimators conducted in Baraud and Birg\'e~\citeyearpar{BarBir2018}[Section~9] for estimating a regression function (seen as the parameter of interest in the conditional distribution of $Y$ given $W$). In Baraud and Birg\'e~\citeyearpar{BarBir2018}, the authors studied a regression model in which the errors are assumed to be i.i.d.,\ homoscedastic with a density with respect to the Lebesgue measure. In the present paper, the errors are typically heteroscedastic, independent but not i.i.d.\ and they may not admit a density with respect {to} the Lebesgue measure. This is the case in the logistic and Poisson regression settings for example. Actually, new results had to be established in order to analyze further the behaviour of $\rho$-estimators in the statistical setting we consider here. The proof of our main result combines the theory of $\rho$-estimation --- see Baraud {\em et al.}~\citeyearpar{MR3595933} and Baraud and Birg\'e ~\citeyearpar{BarBir2018} --- and an original result that establishes the fact that the family of functions on $\sW\times \sY$ of the form
\[
(w,y)\mapsto S(y)\geta(w)-A(\geta(w))\quad \text{with}\quad \geta\in \bs{\Gamma}
\]
is VC-subgraph when $S$ is an arbitrary function on $\sY$, $A$ a convex function defined on an interval $I$ of positive length and $\bs{\Gamma}$ a VC-subgraph class of functions defined on $\sW$ with values in $I$. The proof of our main result also relies on an upper bound with explicit constants (see Theorem~\ref{thm-VCB}) on the expectation of the supremum of an empirical process over a VC-class of functions. Since we are not aware of such a result (with explicit constants) in the literature, this bound can be of independent interest. 

Besides our theoretical guarantees on the performance of the estimator $\widehat \bst$, we carry out a simulation study in order to compare it with the MLE and median-based estimators. The simulation study addresses both the situations where the data are generated from the model and when it is contaminated or contains an outlier. To our knowledge, it is the first time that $\rho$-estimators are implemented numerically and their performance is studied on simulated data. 

There exist only few papers in the literature that tackle the estimation problem {corresponding to the framework described by (\ref{eq-model})} and establish risk bounds for the proposed estimators of $\bst\et$. When $\sW=[0,1]$, Kolaczyk and Nowak~\citeyearpar{MR2158614} proposed an estimation of $\bst\et$ by piecewise polynomials. When the exponential family is given in its canonical form and the natural parameter is a smooth function of the mean, they {build} estimators that achieve, up to extra logarithmic factors, the classical rate $n^{-\alpha/(2\alpha+1)}$ over Besov balls with regularity $\alpha>0$ for a Hellinger-type loss {between the corresponding conditional distributions}. Brown {\em et al.}~\citeyearpar{brown2010} considered one-parameter exponential families which possess the property that the variances of the distributions are quadratic functions of their means. These families include as special cases the binomial, gamma and Poisson distributions, among others, and have been studied earlier by Antoniadis and Sapatinas~\citeyearpar{MR1859411}. When the exponential family is parametrized by its mean, Brown {\em et al.}~\citeyearpar{brown2010} used a variance stabilizing transformation in order to turn the original problem of estimating the function $\bst\et$ into that of estimating a regression function in the homoscedastic Gaussian regression framework. They established uniform rates of convergence with respect to some $\L_{2}$-loss over classes of functions $\bst\et$ that belong to Besov balls and are bounded from above and below by positive numbers. Finally, in the case of the Poisson family parametrized by its mean,  Kroll~\citeyearpar{Kroll-1} proposed an estimator of $\bst\et$ which is based on a  model selection procedure. For some $\L_{2}$-loss, he proved that his estimator achieved the minimax rate of convergence over Sobolev-type ellipsoids. 

A common feature of these papers lies in the fact that they make strong assumptions on the function {$\bst\et$} and the distribution of the covariates $W_{i}$ while our approach does not require any assumption about these quantities. They assume that the distribution of the $W_{i}$ is known or partly known and the  function $\bst\et$ is smooth enough. Besides, none of these papers addresses the problem of model misspecification nor proposes an estimator to solve it.

Our paper is organized as follows. We describe our statistical framework in Section~\ref{sect-1} and present there several examples to which our approach applies. The construction of the estimator and our main result about its risk are presented in Section~\ref{sect-MR}. We  also explain why the deviation inequality we derive guarantees the desired robustness property of the estimator. Uniform risk bounds over H\"olderian classes are established in Section~\ref{sect-URB} {provided that the exponential family involved in the model is suitably parametrized}. We also show that{, without such a suitable parametrization,} the minimax rates may differ from the usual ones established {for an homoscedastic Gaussian regression as described by our Example~\ref{exa-G} below.} Section~\ref{sect-ST} is devoted to the description of our algorithm and the simulation study. Our bound  on the expectation of the supremum of an empirical process over a VC-subgraph class can be found in Section~\ref{sect-EmpProc} as well as its proof. Section~\ref{sect-proof} is devoted to the other proofs. 

\section{The statistical setting}\label{sect-1}
{Let us recall that we} observe $n$ pairs of independent, but not necessarily i.i.d., random variables $X_{1}=(W_{1},Y_{1}),\ldots,X_{n}=(W_{n},Y_{n})$ with values in a measurable product space $(\sX,\cX)=(\sW\times\sY,\cW\otimes\cY)$ {and we assume that,} for each $i\in\{1,\ldots,n\}$, the conditional {distribution} of $Y_{i}$ given {$W_{i}=w_{i}$} exists and is given by the value at {$w_{i}$} of a measurable function $Q_{i}\et$ from $(\sW,\cW)$ to the set $\sT$ of all probabilities on $(\sY,\cY)$. We equip $\sT$ with the Borel $\sigma$-algebra $\cT$ associated to the total variation distance (which induces the same topology as the Hellinger one). {We recall that the Hellinger distance between two probabilities $P=p\cdot \mu$ and $R=r\cdot \mu$  dominated by a measure $\mu$ on a measurable space $(E,\cE)$  is given by  
\[
h(P,R)=\cro{\frac{1}{2}\int_{E}\pa{\sqrt{p}-\sqrt{r}}^{2}d\mu}^{1/2},
\]
the result being independent of the choice of the dominating measure $\mu$. 
With this choice of $\cT$,} the mapping $w\mapsto h^{2}(Q_{i}\et(w),Q)$ on $(\sW,\cW)$ is measurable whatever the probability $Q\in\sT$ and $i\in\{1,\ldots,n\}$.  

Apart from independence of the $W_{i}$, $1\le i\le n$, we assume nothing about their respective distributions $P_{W_{i}}$ which can therefore be arbitrary.

Let $\overline \sQ\subset \sT$ be an exponential family on the measured space $(\sY,\cY,\nu)$ where $\nu$ is an arbitrary $\sigma$-finite (positive) measure. We assume that $\overline \sQ=\{Q_{\theta},\; \theta\in I\}$ is indexed by a natural parameter $\theta$ that belongs to {some interval $I\subset \R$ such that $\mathring{I}\ne \varnothing$}. This means that, for all $\theta\in I$, the distribution $Q_{\theta}$ admits a density (with respect to $\nu$) of the form 
\begin{equation}\label{MEb}
q_{\theta}:y\mapsto e^{S(y)\theta-A(\theta)}\quad \text{with}\quad A(\theta)=\log\cro{\int_{\sY}e^{\theta S(y)}d\nu(y)},
\end{equation}
where $S$ is a real-valued measurable function on $(\sY,\cY)$ which does not coincide with a constant $\nu$-a.e. We also recall  that the function $A$ is infinitely differentiable on $\mathring{I}$ and strictly convex on $I$. It is of course possible to parametrize $\overline \sQ$ in a different way (i.e.\ with a non-natural parameter) by performing a variable change $\gamma=v(\theta)$ where $v$ is a continuous and strictly monotone function on $I$. We shall see in Section~\ref{sect-ChgVar} that our main result remains unchanged under such a transformation and we therefore choose, for the sake of simplicity, to introduce it under a natural parametrization first. 

Given a class of functions $\bsT$ from $\sW$ into $I$, we presume that there exists $\gtheta\et$ in $\bsT$ such that the conditional distribution {$Q_{i}\et(w_{i})$ is of the form $Q_{\gtheta\et(w_{i})}$ for all $i\in\{1,\ldots,n\}$ and $w_{i}\in\sW$}. We refer to $\bst\et$ as the {\em regression function}. 
Even though our estimator is based on these assumptions, {we should keep in mind that our statistical model might be misspecified: the conditional distributions $Q_{i}\et(w_{i})$ might not be exactly of the form $Q_{\gtheta\et(w_{i})}$, the set $\bsT$ might not contain $\bst\et$ or some observations might be outliers. It will follow from our risk bounds as described by Theorem~\ref{thm-1}
that such misspecifications result in an additional term in the risk corresponding to the approximation error between the truth and the model. This term is small when our model provides a good enough approximation of the truth.}

For $i\in\{1,\ldots,n\}$, let $\sQ_{\sW}$ be the set of all measurable mappings (conditional probabilities) from $(\sW,\cW)$ into $(\sT,\cT)$. We set $\sbQ_{\sW}=\sQ_{\sW}^{n}$ so that the $n$-tuple $\gQ\et=(\etc{Q\et})$ belongs to $\sbQ_{\sW}$ {as well as the $n$-tuple $\gQ_{\bst}=(Q_{\bst},\ldots,Q_{\bst})$ where $Q_{\bst}\in\sQ_{\sW}$ denotes the mapping $w\mapsto Q_{\bst(w)}$ when  $\bst$ is a measurable function from $\sW$ into $I$}. We endow the space $\sbQ_{\sW}$ with the Hellinger-type (pseudo) distance $\gh$ defined as follows. For $\gQ=(Q_{1},\ldots,Q_{n})$ and $\gQ'=(Q_{1}',\ldots,Q_{n}')$ in $\sbQ_{\sW}$, 
\begin{align}
\gh^{2}(\gQ,\gQ')&=\E\cro{\sum_{i=1}^{n}h^{2}\pa{Q_{i}(W_{i}),Q_{i}'(W_{i})}}\label{eq-loss}\\
&=\sum_{i=1}^{n}\int_{\sW}h^{2}\pa{Q_{i}(w),Q_{i}'(w)}dP_{W_{i}}(w).\nonumber
\end{align}
In particular, $\gh(\gQ,\gQ')=0$ implies that for all $i\in\{1,\ldots,n\}$, $Q_{i}=Q_{i}'$ $P_{W_{i}}$-a.s. 

On the basis of the observations $\etc{X}$, we  build an estimator $\widehat \bst$ of $\bst\et$ with values in $\bsT$ and evaluate its performance by the quantity  
\[
\gh^{2}(\gQ\et,\gQ_{\widehat{\gtheta}})=\sum_{i=1}^{n}\int_{\sW}h^{2}\pa{Q_{i}\et(w),Q_{\widehat{\gtheta}(w)}}dP_{W_{i}}(w).
\]

{When $P$ is the distribution of a random variable $(W,Y)\in\sW\times \sY$ we write it as $P=Q\cdot P_{W}$ where $P_{W}$ is the marginal distribution of $W$ and $Q$ the conditional distribution of $Y$ given $W$}. For $P=Q\cdot P_{W}$ and $P'=Q'\cdot P_{W}$ the squared Hellinger distance between $P$ and $P'$ is written as 
\[
h^{2}(P,P')=\int_{\sW}h^{2}\pa{Q(w),Q'(w)}dP_{W}(w).
\]
Setting, for $i\in\{1,\ldots,n\}$ and {$\bst$ a function from $\sW$ to $I$},  $P_{i}\et=Q_{i}\et\cdot P_{W_{i}}$ and $P_{i,\bst}=Q_{\bst}\cdot P_{W_{i}}$  we deduce that 
\[
\gh^{2}(\gQ\et,\gQ_{\bst})=\sum_{i=1}^{n}h^{2}\pa{P_{i}\et,P_{i,\bst}}
\]
{so that $\gh^{2}(\gQ\et,\gQ_{\bst})$ is equal to} $\gh^{2}(\gP\et,\gP_{\bst})=\sum_{i=1}^{n}h^{2}\pa{P_{i}\et,P_{i,\bst}}$ where $\gP\et=\otimes_{i=1}^{n}P_{i}\et$ is the true distribution of the observed data $\bsX=(X_{1},\ldots,X_{n})$ while $\gP_{\bst}=\bigotimes_{i=1}^{n}P_{i,\bst}=\bigotimes_{i=1}^{n}(Q_{\bst}\cdot P_{W_{i}})$ is the joint distribution of independent random variables {$(W'_{1},Y'_{1}),\ldots,(W'_{n},Y'_{n})$ for which the conditional distribution of $Y'_{i}$ given $W'_{i}=w_{i}$ is given by $Q_{\bst(w_{i})}\in\overline \sQ$ for all $i$. This shows that the quantity $\gh(\gQ\et,\gQ_{\bst})=\gh(\gP\et,\gP_{\bst})$ may also be interpreted as a distance between the probability distributions $\gP\et$ and $\gP_{\bst}$ and not only as a (pseudo) distance between the conditional ones $\gQ\et$ and $\gQ_{\bst}$. More generally, given two measurable functions $\bst,\bst'$ {from $\sW$ to $I$}, the quantity $\gh(\gQ_{\bst},\gQ_{\bst'})$ can also be written as $\gh(\gP_{\bst},\gP_{\bst'})$. Note that, unlike $\gQ_{\widehat \bst}$, $\gP_{\widehat \bst}$ is not an estimator (of $\gP\et$) since it depends on the marginal distributions $P_{W_{1}},\ldots,P_{W_{n}}$ which are unknown.

\subsection{Examples}
Let us present here some typical statistical {models} to which our approach applies.
\begin{exa}[{Homoscedastic} Gaussian regression with known variance]\label{exa-G}
Given $n$ independent random variables $W_{1},\ldots,W_{n}$ with values in $\sW$, let
\[
Y_{i}=\gtheta\et(W_{i})+\sigma\eps_{i}\quad \text{for all $i\in\{1,\ldots,n\}$},
\]
where the $\eps_{i}$ are i.i.d.\ standard real-valued Gaussian random variables, $\sigma$ is a known positive number and $\gtheta\et$ an unknown regression function with values in $I=\R$. In this case, $\overline \sQ$ is the set of all Gaussian distributions with variance $\sigma^{2}$ and for all $\theta\in I=\R$, $Q_{\theta}=\cN(\theta,\sigma^{2})$ has a density with respect to $\nu=\cN(0,\sigma^{2})$ on $(\sY,\cY)=(\R,\cB(\R))$ which is of the form~\eref{MEb} with $A(\theta)=\theta^{2}/(2\sigma^{2})$ and $S(y)=y/\sigma^{2}$ for all $y\in\R$.
\end{exa}
\begin{exa}[Binary regression]\label{exa-L}
The pairs of random variables $(W_{i},Y_{i})$ with $i\in\{1,\ldots,n\}$ are independent with values in $\sW\times\{0,1\}$ and
\begin{equation}\label{mod-LR}
{\P\cro{Y_{i}=y|W_{i}=w_{i}}=\frac{\exp\cro{y\bst\et(w_{i})}}{1+\exp\cro{\bst\et(w_{i})}}\quad \text{for all }y\in\{0,1\}\text{ and }w_{i}\in\sW.}
\end{equation}
This means that the conditional distribution of $Y_{i}$ given {$W_{i}=w_{i}$ is Bernoulli with mean $\pa{1+\exp\cro{-\bst\et(w_{i})}}^{-1}$} for some regression function $\bst\et$ with values in $I=\R$. This model is equivalent to the logit one presented in Example~\ref{exa-C} by changing $Y_{i}\in\{0,1\}$ into $Y_{i}'=2Y_{i}-1\in\{-1,1\}$ for all $i$.  The exponential family $\overline\sQ$ consists of the Bernoulli distribution $Q_{\theta}$ with mean $1/[1+e^{-\theta}]\in (0,1)$ and $\theta\in I=\R$. For all $\theta\in\R$, $Q_{\theta}$ admits a density with respect to the counting measure $\nu$ on $\sY=\{0,1\}$ of the form~\eref{MEb} with $A(\theta)=\log(1+e^{\theta})$ and $S(y)=y$ for all $y\in\sY$. 
\end{exa}
\begin{exa}[Poisson regression]\label{exa-P}
The exponential family $\overline \sQ$ is the set of all Poisson distributions $Q_{\theta}$ with mean $e^{\theta}$, $\theta\in I=\R$. Taking for $\nu$ the Poisson distribution with mean 1, the density of $Q_{\theta}$ with respect to $\nu$ takes the form~\eref{MEb} with $S(y)=y$ for all $y\in\N$ and $A(\theta)=e^{\theta}-1$ for all $\theta\in\R$. The conditional distribution of $Y_{i}$ given {$W_{i}=w_{i}$ is presumed to be Poisson with mean $\exp\cro{\bst\et(w_{i})}$} for some regression function $\bst\et$ with values in $I=\R$.
\end{exa}
\begin{exa}[Exponential multiplicative regression]\label{exa-E}
The random variables $W_{1},\ldots,W_{n}$ are independent and 
\begin{equation}\label{eq-expmod}
Y_{i}=\frac{Z_{i}}{\gtheta\et(W_{i})}\quad \text{for all $i\in\{1,\ldots,n\}$}
\end{equation}
where the $Z_{i}$ are i.i.d.\ with exponential distribution of parameter 1 and independent of the $W_{i}$. The conditional distribution of $Y_{i}$ given $W_{i}=w_{i}$ is then exponential with mean $1/\bst\et(w_{i})\in I=(0,+\infty)$. Exponential distributions parametrized by $\theta\in I$ admit densities with respect to the Lebesgue measure on $\R_{+}$ of the form~\eref{MEb} with $S(y)=-y$ for all $y\in\sY=\R_{+}$ and $A(\theta)=-\log \theta$.
\end{exa}
\section{The main results}\label{sect-MR}
\subsection{The estimation procedure}\label{Sect-def-est}
As mentioned in the introduction, our approach is based on $\rho$-estimation. We do not recall here the basic ideas that underline the construction of these estimators and rather refer the reader to Baraud and Birg\'e~\citeyearpar{BarBir2018}. Let $\psi$ be the function defined on $[0,+\infty]$ by
\begin{equation}\label{def-psi}
\psi(x)=\frac{x-1}{x+1}\ \quad\ \mbox{for $x\in[0,+\infty)$}\ \quad\ \mbox{and}\ \quad \psi(+\infty)=1.
\end{equation}
{Let us set, for $\bst\in\bsT$, $q_{\bst}(X_{i})=q_{\bst(W_{i})}(Y_{i})$, where $q_{\theta}$ is given by \eref{MEb} and, in order to avoid measurability issues, let us restrict ourselves to those $\bst$ belonging to a finite or countable subset $\bsTT$ of $\bsT$. We then introduce the function}
\begin{equation}\label{def-T}
\gT(\bsX,\bst,\bst')=\sum_{i=1}^{n}\psi\pa{\sqrt{\frac{q_{\bst'}(X_{i})}{q_{\bst}(X_{i})}}}\quad\mbox{for }\ \quad \bst,\bst'\in\bsTT,
\end{equation}
with the conventions $0/0=1$ and $a/0=+\infty$ for all $a>0$. We set 
\begin{equation}\label{def-gup}
\gup(\bsX,\bst)=\sup_{\bst'\in\bsTT}\gT(\bsX,\bst,\bst')\ \quad\ \mbox{for all}\ \quad\ \bst\in\bsTT
\end{equation}
{and choose $\widehat \bst=\widehat \bst(\bsX)$ as any (measurable) element of the random (and non-void) set 
\begin{equation}\label{def-sE}
\sE(\bsX)=\ac{\bst\in \bsTT\;\mbox{ such that }\; \gup(\bsX,\bst)\le \inf_{\bst'\in\bsTT}\gup(\bsX,\bst')+\frac{\kappa}{25}}
\end{equation}
with $\kappa=280\sqrt{2}+74$, {so that} $18<\kappa/25<18.8$. The random variable $\widehat \bst(\bsX)$ is our estimator of the regression function $\bst\et$ and $\gQ_{\widehat \bst}=(Q_{\widehat \bst},\ldots,Q_{\widehat \bst})$.}

{Note that the} construction of the estimator is only based on the choices of the exponential family given by~\eref{MEb} and the subset $\bsTT$ of $\bsT$. In particular, the estimator does not depend on the distributions $P_{W_{i}}$ of the $W_{i}$ which may therefore be unknown.  


The fact that we build our estimator on a finite or countable subset $\bsTT$ of $\bsT$ is not restrictive as we shall see. Besides, this assumption is consistent with the practice of calculating an estimator on a computer that can  handle a finite number of values only. 

{Let us mention that similar results could be established for the $\rho$-estimator associated to the alternative choice $\psi(x)=(x-1)/\sqrt{x^{2}+1}$. Nevertheless, the risk bounds we get for this choice of $\psi$ 
 involve numerical constants that are larger than those we establish here for $\psi(x)=(x-1)/(x+1)$. In the present paper, we therefore focus on this latter choice of $\psi$.}

\subsection{The main assumption and  the performance of $\widehat \bst$}
{Let us assume the following}: 
\begin{ass}\label{Hypo}
The class {of functions} $\bsT$ is VC-subgraph on $\sW$ with dimension not larger than $V\ge 1$.
\end{ass}
We refer the reader to {Section~2.6.2 of van der Vaart and Wellner~\citeyearpar{MR1385671} or Section~8 of Baraud et al.~\citeyearpar{MR3595933}} for the definition of VC-subgraph classes and their properties. In the present paper, we mainly use the facts that Assumption~\ref{Hypo} is satisfied when $\bsT$ is a linear space $\cV$ with finite dimension $d\ge 1$, in which case $V=d+1$ and that it is also satisfied when $\bsT$ is of the form $\{F(\bsb),\; \bsb\in\cV\}$ where $F$ is a monotone  function on the real-line. In this latter case, the VC-dimension of $\bsT$ is not larger than that of $\cV$. 
%
%
%
%
We set 
\begin{equation}\label{eq-const}
c_{1}=150,\quad c_{2}=1.1\times 10^{6},\quad c_{3}=5014 
\end{equation}
and, for $\gQ\in\sbQ_{\sW}$ and $\gA\subset \sbQ_{\sW}$, 
\[
\gh\pa{\gQ,\gA}=\inf_{\gQ'\in\gA}\gh\pa{\gQ,\gQ'}.
\]
{The  following theorem provides a probabilistic bound for the distance between our estimator $\gQ_{\widehat \bst}$ and $\gQ\et$.}
\begin{thm}\label{thm-1}
Let $\xi>0$. Under Assumption~\ref{Hypo}, whatever the conditional probabilities  $\gQ\et=(Q_{1}\et,\ldots,Q_{n}\et)$ of the $Y_{i}$ given $W_{i}$ and the distributions of the $W_{i}$, the estimator $\gQ_{\widehat \bst}$ defined in Section~\ref{Sect-def-est} satisfies, with a probability at least $1-e^{-\xi}$,
\begin{align}
\gh^{2}\pa{\gQ\et,\gQ_{\widehat \bst}}&\le c_{1}\gh^{2}(\gQ\et,\sbQ)+c_{2}V\cro{9.11+\log_{+}\pa{\frac{n}{V}}}+c_{3}\pa{1.5+\xi}\label{eq-1mod}
\end{align}
where $\sbQ=\{\gQ_{\bst}=(Q_{\bst},\ldots,Q_{\bst}),\; \bst\in \bsTT\}$ and $\log_{+}=\max(0,\log)$.
\end{thm}

The constants $c_{1},c_{2}$ and $c_{3}$ are numerical constants. They are independent of the choice of the exponential family. When the model $\sbQ$ is exact, the bound we get only depends on the VC-dimension of $\bsT$.

It is clear that~\eref{eq-1mod} also holds true for $\overline \sbQ=\{\gQ_{\bst},\; \bst\in \bsT\}$ in place of $\sbQ$  when $\sbQ$ is dense in $\overline \sbQ$ with respect to the Hellinger-type distance $\gh$. This is the case when $\bsTT$ is dense in $\bsT$ for the topology of pointwise convergence. We do not comment on our result any further in this direction and rather refer to Baraud and Birg\'e~\citeyearpar{BarBir2018} Section~4.2. From now on, we assume for the sake of simplicity that $\sbQ$ is dense in $\overline \sbQ$, doing as if $\bsT=\bsTT$. In the remaining part of this section, $C$ will denote a positive numerical constant that may vary from line to line.

Let us now {rewrite}~\eref{eq-1mod} in a slightly different form. We have seen in Section~\ref{sect-1} that the quantity $\gh\pa{\gQ\et,\gQ_{\bst}}$ with $\bst\in\bsT$, which involves the conditional probabilities of $\gP\et$ and $\gP_{\bst}$ with respect to the $W_{i}$, can also be interpreted in terms of the Hellinger(-type) distance between these two product probabilities. Inequality~\eref{eq-1mod} therefore implies that
\begin{equation}
\P\left[C\gh^{2}\pa{\gP\et,\gP_{\widehat \bst}}>\gh^{2}(\gP\et,\overline\sbP)+V\cro{1+\log_{+}\pa{\frac{n}{V}}}+\xi\right]\le e^{-\xi}, \label{eq-1mod2b}
\end{equation}
where $\overline\sbP=\{\gP_{\bst},\; \bst\in\bsT\}$. Integrating this inequality with respect to $\xi>0$ leads to the following risk bound for our estimator $\widehat \bst$
\begin{equation}\label{eq-1mod3}
C\E\cro{\gh^{2}\pa{\gP\et,\gP_{\widehat \bst}}}\le \gh^{2}(\gP\et,\overline\sbP)+V\cro{1+\log_{+}\pa{\frac{n}{V}}}.
\end{equation}
%
%

In order to comment upon~\eref{eq-1mod3}, let us start with the ideal situation where $\gP\et$ belongs to $\overline \sbP$, i.e.\ $\gP\et=\gP_{\bst\et}$ {for some} $\bst\et\in \bsT$, in which case~\eref{eq-1mod3} leads to
\begin{equation}\label{eq-BRex}
C\E\cro{\gh^{2}\pa{\gP_{\bst\et},\gP_{\widehat \bst}}}\le V\cro{1+\log_{+}\pa{\frac{n}{V}}}.
\end{equation}
Up to the logarithmic factor, the right-hand side of this inequality is of the expected order of magnitude $V$ for the quantity $\gh^{2}(\gP_{\bst\et},\gP_{\widehat \bst})$: in typical situations $V$ is of the same order as the number of parameters that are used to parametrize $\bsT$.

When the true distribution $\gP\et$ {is of the form} $\gP_{\bst\et}$ but the regression function $\bst\et$ does not belong to $\bsT$, or when the conditional distributions of the $Y_{i}$ given $W_{i}$ do not belong to our exponential family, inequality~\eref{eq-1mod3} shows that, as compared to~\eref{eq-BRex}, the bound we get involves the approximation term $\gh^{2}(\gP\et,\overline\sbP)$ that accounts for the fact that our statistical model is misspecified. However, as long as this quantity remains small enough as compared to $V\cro{1+\log_{+}(n/V)}$, our risk bound will be of the same order as that given by~\eref{eq-BRex} when the model is exact. This property accounts for the stability of our estimation procedure under misspecification. In order to be more specific, let us assume that 
\begin{equation}\label{def-Pstar}
\gP\et=\bigotimes_{i=1}^{n}\left[(1-\alpha_{i})P_{i,\overline \bst}+\alpha_{i} R_{i}\right]\quad{\text{and}}\quad\sum_{i=1}^{n}\alpha_{i}\le \frac{n}{2}
\end{equation}
where {$\overline \bst\in\bsT$}, $R_{i}$ is an arbitrary distribution on $\sX$ and $\alpha_{i}$ a number in $[0,1]$ for all $i\in\{1,\ldots,n\}$. Such a distribution $\gP\et$ allows us to model different form of robustness including robustness to the presence of contaminating data as well as outliers.  
In the case of contamination, $P_{W_{i}}=P_{W}$, $\alpha_{i}=\alpha\in (0,1/2]$, $R_{i}=R\ne P_{\overline \bst}=Q_{\overline \bst}\cdot P_{W}$ for all $i\in\{1,\ldots,n\}$ and one observes an $n$-sample a portion $(1-\alpha)$ of which is drawn according to a distribution $P_{\overline \bst}$ that belongs to our model $\overline \sP=\{P_{\bst}=Q_{\bst}\cdot P_{W},\; \bst\in \bsT\}$ while the remaining part of the data is drawn according to a contaminating distribution $R$. In the second case, the data set contains the outliers $\{a_{i},\; i\in K\}$ for some subset $K\subset \{1,\ldots,n\}$ with $K\ne \varnothing$ so that $\gP\et$ is of the form~\eref{def-Pstar} with $\alpha_{i}=\1_{i\in K}$ for all $i\in\{1,\ldots,n\}$ and $R_{i}=\delta_{a_{i}}$ for all $i\in K$.  In all cases, using the classical inequality $h^{2}\le D$ where $D$ denotes the total variation distance between probabilities, we get
\begin{equation}\label{eq-biais}
\gh^{2}(\gP\et,\overline\sbP)\le \gh^{2}(\gP\et,\gP_{\overline \bst})\le \sum_{i=1}^{n}D(P_{i}\et, P_{\overline \bst})\le \sum_{i=1}^{n} \alpha_{i},
\end{equation}
which means that whenever $\sum_{i=1}^{n} \alpha_{i}$ remains small as compared to $V(1+\log_{+}(n/V))$, the performance of the estimator remains almost the same as if $\gP\et$ were equal to $\gP_{\overline \bst}$. The estimator $\widehat \bst$ therefore possesses some stability properties with respect to contamination and the presence of outliers. 

\subsection{From a natural to a general exponential family}\label{sect-ChgVar}
In Section~\ref{sect-1}, we focused on an exponential family $\overline \sQ$ parametrized by its natural parameter. However statisticians often write exponential families $\overline \sQ$ under the general form $\overline \sQ=\{R_{\gamma}=r_{\gamma}\cdot\nu,\; \gamma\in J\}$ with
%
\begin{equation}\label{MEg}
r_{\gamma}: y\mapsto e^{u(\gamma)S(y)-B(\gamma)} \quad \text{for $\gamma\in J$}.
\end{equation}
In~\eref{MEg}, $J$ denotes a (non-degenerate) interval of $\R$ and $u$ a continuous and strictly monotone function from $J$ onto $I$ so that $B=A\circ u$. In the exponential family $\overline \sQ=\{R_{\gamma},\; \gamma\in J\}=\{Q_{\theta},\; \theta\in I\}$, the probabilities $R_{\gamma}$ are associated to the probabilities $Q_{\theta}$ by the formula $R_{\gamma}=Q_{u(\gamma)}$.

With this new parametrization, we could alternatively write our statistical model $\overline \sbQ$ as 
\begin{equation}\label{eq-modelG}
\overline \sbQ=\ac{\gR_{\bsg}=(R_{\bsg},\ldots,R_{\bsg}),\; \bsg\in \bsG}
\end{equation}
where $\bsG$ is a class of  functions $\bsg$ from $\sW$ into $J$. Starting from such a statistical model and presuming that $\gQ\et=\gR_{\bsg\et}$ for some function $\bsg\et\in\bsG$, we can build an estimator $\widehat \bsg$ of $\bsg\et$ as follows: given a finite or countable subset $\gGamma$ of $\bsG$ we set $\widehat \bsg=u^{-1}(\widehat \bst)$ where $\widehat \bst$ is any estimator obtained by applying the procedure described in Section~\ref{Sect-def-est} under the natural parametrization of the exponential family $\overline \sQ$ and using the finite or countable model $\bsTT=\{\bst=u\circ \bsg,\; \bsg\in\gGamma\}$.  

Since our model $\overline \sbQ$ for the conditional probabilities $\gQ\et$ is unchanged (only its parametrization changes), it would be interesting to establish a result on the performance of the estimator $\gR_{\widehat \bsg}=\gQ_{\widehat \bst}$ which is independent of the parametrization. A nice feature of the VC-subgraph property lies in the fact that it is preserved by composition with a monotone function: since $u$ is monotone, if $\bsG$ is VC-subgraph with dimension not larger than $V$, so is $\bsT$ and our Theorem~\ref{thm-1} applies. 
%
The following corollary is therefore straightforward. 
\begin{cor}\label{cor-general}
Let $\xi>0$. If the statistical model $\overline \sbQ$ is under the general form~\eref{eq-modelG} and $\bsG$  is VC-subgraph with dimension not larger than $V\ge 1$, whatever the conditional probabilities  $\gQ\et=(Q_{1}\et,\ldots,Q_{n}\et)$ of the $Y_{i}$ given $W_{i}$ and the distributions of the $W_{i}$, the estimator $\gR_{\widehat \bsg}$ satisfies with a probability at least $1-e^{-\xi}$,
\begin{align}
\gh^{2}\pa{\gQ\et,\gR_{\widehat \bsg}}&\le c_{1}\gh^{2}(\gQ\et,\sbQ)+c_{2}V\cro{9.11+\log_{+}\pa{\frac{n}{V}}}+c_{3}\pa{1.5+\xi}\label{eq-1modG}
\end{align}
where $\sbQ=\{\gR_{\bsg},\; \bsg\in \gGamma\}$. In particular, 
\begin{equation}\label{eq-risk}
\E\cro{\gh^{2}\pa{\gQ\et,\gR_{\widehat \bsg}}}\le C'\cro{\gh^{2}(\gQ\et,\sbQ)+V\cro{1+\log_{+}\pa{\frac{n}{V}}}},
\end{equation}
for some numerical constant $C'>0$.
\end{cor}
A nice feature of our approach lies in the fact that~\eref{eq-1modG} holds for all exponential families simultaneously and all ways of parametrizing them. In particular, the VC-dimension associated to the model $\overline \sbQ$ is intrinsic since it is independent of the way it is parametrized.

\section{Uniform risk bounds}\label{sect-URB}
Throughout this section, we assume that the $W_{i}$ are i.i.d.\ with common distribution $P_{W}$ and that {$\gQ\et=\gR_{\bsg\et}=(R_{\bsg\et},\ldots,R_{\bsg\et})$} belongs to a statistical model of the (general) form given by~\eref{eq-modelG} where $\bsG$ is a class of smooth functions. More precisely, we assume that for some $\alpha\in (0,1]$ and $M>0$, $\bsG=\cH_{\alpha}(M)$ is  the set of functions $\bsg$ on {$\sW=[0,1]$} with values in $J$ that satisfy the H\"older condition
\begin{equation}\label{def-Hold}
\ab{\bsg(x)-\bsg(y)}\le M|x-y|^{\alpha}\quad\text{for all $x,y\in [0,1]$}
\end{equation}
{and denote by $\overline \sQ[\cH_{\alpha}(M)]$ the corresponding family of conditional distributions $Q\et=R_{\bsg\et}$. Because of our equidistribution assumption and the  form of our loss function, the loss $\gh^{2}\pa{\gQ\et,\gR_{\widetilde \bsg}}$ takes the form $nh^{2}\pa{Q\et,R_{\widetilde \bsg}}$ whatever the estimator $\widetilde \bsg$.} 

Our aim is both to estimate {$\gR_{\bsg\et}$, or equivalently $R_{\bsg\et}$, under the assumption that $\bsg\et\in\cH_{\alpha}(M)$ and to evaluate the minimax risk over $\overline \sQ[\cH_{\alpha}(M)]$}, i.e.\  the quantity 
\begin{equation}
\cR_{n}(\cH_{\alpha}(M))=\adjustlimits\inf_{\widetilde \bsg}\sup_{\bsg\et\in\cH_{\alpha}(M)}\E\cro{h^{2}\pa{R_{\bsg\et},R_{\widetilde \bsg}}}
\label{def-Minimax}
\end{equation}
with
\[
h^{2}\pa{R_{\bsg},R_{\bsg'}}\eqd\int_{\sW}h^{2}\pa{R_{\bsg(w)},R_{\bsg'(w)}}dP_{W}(w).
\]
where the infimum runs among all estimators $\widetilde \bsg$ of $\bsg\et$ based on the $n$-sample $\etc{X}$. 
%
%

\subsection{Uniform risk bounds over H\"older classes}
It is common to parametrize the exponential family $\overline \sQ=\{R_{\gamma},\; \gamma\in J\}$ by the means of the distributions, i.e.\ with $\gamma=\int_{\sY}ydR_{\gamma}(y)$. This is typically the case for the Bernoulli, Gaussian and Poisson families for example. In such a case, one can write the model in a heteroscedastic regression form:
\begin{equation}
Y_{i}=\bsg\et(W_{i})+\sigma\left(\st\bsg\et(W_{i})\right)\eps_{i}\quad \text{for all $i\in\{1,\ldots,n\}$},
\label{eq-heteroreg}
\end{equation}
where $\sigma^{2}\left(\st\bsg\et(W_{i})\right)$ is the conditional variance of $Y_{i}$ and the $\varepsilon_{i}$ errors which, conditionally to the $W_{i}$ are centered and with variance 1. As mentioned in the introduction, many authors have used this form and its similarity to the classical Gaussian regression framework given in Example~\ref{exa-G} to derive estimators with performances that mimic those of the least squares estimators in the Gaussian case. In particular, when $\bsg\et\in\cH_{\alpha}(M)$, the minimax rate for Gaussian regression is $n^{-2\alpha/(2\alpha+1)}$ and, for the more general situation described in (\ref{eq-heteroreg}) various authors established similar rates for their estimators by using the $\L_{2}$-loss (under somewhat restrictive assumptions  as mentioned in the introduction). 

With our Hellinger-type loss, we also show in this section that $n^{-2\alpha/(2\alpha+1)}$ is the minimax rate  for estimating $R_{\ggamma\et}$ with $\ggamma\et\in \cH_{\alpha}(M)$. However, this result holds when the parametrization of the exponential family satisfies some suitable conditions. When these conditions are not met, the minimax rate may be different as we shall see, even when the exponential family is parametrized by the mean as it is commonly done in the literature. In any case,  our estimator is proven to achieve the minimax rate up to a logarithmic factor.  
%
%

Let us first introduce the following assumptions on the parametrization.
\begin{ass}\label{H-param}
There exists a constant $\kappa>0$ such that 
\begin{equation}\label{eq-bsph}
h(R_{\gamma},R_{\gamma'})\le \kappa\ab{\gamma-\gamma'}\quad \text{for all $\gamma,\gamma'\in J$}
\end{equation}
and for a (non-degenerate) compact interval $K\subset J$, there exists a constant $c_{K}>0$ such that 
\begin{equation}\label{eq-bifh}
h(R_{\gamma},R_{\gamma'})\ge c_{K}\ab{\gamma-\gamma'}\quad \text{for all $\gamma,\gamma'\in K$.}
\end{equation}
\end{ass}
This assumption is in particular satisfied in the following situation.
\begin{prop}\label{prpL2-h}
Let $\overline \sQ=\{Q_{\theta},\; \theta\in I\}$ be a natural exponential family defined by~\eref{MEb} where $I$ is an open interval. If the function $v$  satisfies
\begin{equation}\label{eq-condv}
v'(\theta)=\sqrt{\frac{A''(\theta)}{8}}>0\quad \text{for all $\theta\in I$},
\end{equation}
when parametrized by  $\gamma=v(\theta)$, the exponential family $\overline \sQ=\{R_{\gamma}=Q_{v^{-1}(\gamma)},\; \gamma\in J\}$ satisfies Assumption~\ref{H-param} with $\kappa=1$ for all choices of a (non-trivial) compact subset $K$ of $J$.
\end{prop}
It is well-known that the functions $v_{j}(\theta)$, $j\in\{1,2,3,4\}$ given by
\[
v_{1}(\theta)=\frac{\theta}{\sigma\sqrt{8}},\;\; v_{2}(\theta)=\frac{1}{\sqrt{2}}\arcsin\left(\frac{1}{\sqrt{1+e^{-\theta}}}\right),\;\; v_{3}(\theta)=\frac{1}{\sqrt{2}}e^{\theta/2}\quad\text{on }\R
\]
and $v_{4}(\theta)=8^{-1/2}\log \theta$ on $(0,+\infty)$  satisfy \eref{eq-condv} in the cases of  Examples~\ref{exa-G}, ~\ref{exa-L}, ~\ref{exa-P} and~\ref{exa-E} respectively. 

As a consequence of Assumption~\ref{H-param} we derive by integration with respect to $P_{W}$ that, for all functions $\bsg,\bsg'$ on $\sW$ with values in~$J$, 
%
\begin{equation}\label{eq-BHEL}
h^{2}\pa{R_{\bsg},R_{\bsg'}}\le \kappa^{2}\norm{\bsg-\bsg'}_{2}^{2}=\kappa^{2}\int_{\sW}\pa{\bsg-\bsg'}^{2}dP_{W},
\end{equation}
which leads to the following uniform risk bound for the performance of the $\rho$-estimator $R_{\widehat \bsg}$ when $\bsg\et\in\cH_{\alpha}(M)$. Note that this upper bound holds without any assumption on $P_{W}$.

\begin{prop}\label{prop-bs}
Assume that~\eref{eq-bsph} is satisfied. Let $\alpha\in (0,1]$, $M>0$ and $\overline \bsS$ be the set of functions with values in the interval $J$ which are piecewise constant on each element of a partition $\{I_{j}, j\in\{1,\ldots,D\}\}$ of $[0,1]$ into $D\ge 1$ intervals of lengths $1/D$. For  
\[
D=D(\alpha,M,n)=\min\ac{k\in\N,\; \pa{\frac{\kappa^{2} M^{2}n}{1+\log n}}^{\frac{1}{1+2\alpha}}\le k},
\]
 the $\rho$-estimator $\widehat \bsg$ based on (any) countable and dense subset $\bsS$ of $\overline \bsS$ (with respect to the supremum norm) satisfies 
\begin{align}
\sup_{\bsg\et\in \cH_{\alpha}(M)}\!\!\E\cro{h^{2}\pa{R_{\bsg\et},R_{\widehat \bsg}}}\le 2C'\cro{\pa{\frac{(\kappa M)^{1/\alpha}\log(en)}{n}}^{\frac{2\alpha}{1+2\alpha}}+\frac{3\log(en)}{2n}}\nonumber
\end{align}
where $C'$ is the numerical constant appearing in~\eref{eq-risk}.
\end{prop}
To show that this rate is optimal under Assumption~\ref{H-param} and the $\rho$-estimator minimax (up to a logarithmic factor) when the distribution of the $W_{i}$ can be arbitrary, let us assume that $P_{W}$ is the uniform distribution on $\sW=[0,1]$. It then follows from (\ref{eq-bifh}) that there exists a constant $c_{K}>0$ such that for all functions $\bsg,\bsg'$ on $\sW$ with values in $K$, 
\[
h^{2}\pa{R_{\bsg},R_{\bsg'}}\ge c_{K}^{2}\norm{\bsg-\bsg'}_{2}^{2}.
\]
Assumption~\ref{H-param} makes the Hellinger-type distance $h\pa{R_{\bsg},R_{\bsg'}}$  and the $\L_{2}(P_{W})$-one between $\bsg$ and $\bsg'$ comparable, at least when $\bsg$ and $\bsg'$ take their values in $K$. 
\begin{prop}\label{prop-bi}
Let $\alpha\in (0,1]$ and $M>0$. If $P_{W}$ is the uniform distribution on $[0,1]$ and Assumption~\ref{H-param} is satisfied for a compact interval $K$ of length $2\overline L>0$, then
\[
\cR_{n}(\cH_{\alpha}(M))\ge  \frac{c_{K}^{2}}{48}\cro{\pa{\frac{3M^{1/\alpha}}{2^{2\alpha+4+1/\alpha}\kappa^{2}n}}^{\frac{2\alpha}{1+2\alpha}}\bigwedge \pa{\frac{M^{2}}{4}}\bigwedge \overline L^{2}}.
\]
\end{prop}
This result says that in all exponential families for which Assumption~\ref{H-param} is satisfied, the order of magnitude of the minimax rate over $\cH_{\alpha}(M)$ cannot be smaller than  $n^{-2\alpha/(2\alpha+1)}$, at least when $P_{W}$ is the uniform distribution on $[0,1]$. 

\subsection{A counterexample}
Without a suitable parametrization of the exponential family $\overline \sQ=\{Q_{\theta},\; \theta\in I\}$ like that provided by Proposition~\ref{prpL2-h}, the minimax rate of convergence of $\cR_{n}(\cH_{\alpha}(M))$ may be different from  $n^{-2\alpha/(2\alpha+1)}$ as shown by the following simple example of Poisson distributions parametrized by their means.
%
\begin{prop}\label{prop-BI0}
Let $\alpha\in (0,1]$, $M>0$, $P_{W}$ be the uniform distribution on $[0,1]$ and $\overline \sQ$ the set of Poisson distributions $R_{\gamma}$ with means $\gamma\in J=(0,+\infty)$. For all $n\ge 1$,  
\[
\cR_{n}(\cH_{\alpha}(M))\ge\frac{(1-e^{-1})}{144}\cro{\pa{\frac{3M^{1/\alpha}}{2^{4+\alpha+3/\alpha}n}}^{\frac{\alpha}{1+\alpha}}\bigwedge \frac{M}{8}\bigwedge \pa{1+\frac{\sqrt{3}}{2}}}.
\]
\end{prop}
In the Poisson case with this parametrization, the rate for $\cR_{n}(\cH_{\alpha}(M))$  is therefore at least of order $n^{-\alpha/(1+\alpha)}$, hence much slower than the one we would get if the family would be properly parametrized as indicated in the previous section, namely $n^{-2\alpha/(2\alpha+1)}$. We conclude that, depending on the exponential family, the parametrization by the mean may lead to different minimax rates. Nevertheless, as shown in the following proposition, the $\rho$-estimator still achieves the optimal rate (up to a logarithmic factor) in this case.
\begin{prop}\label{prop-rhoEstPoisson}
Let $\alpha\in (0,1]$, $M>0$ and $\overline\bsS$ be the set of functions with values in $J=(0,+\infty)$ which are piecewise constant on each element of a partition $\{I_{j}, j\in\{1,\ldots,D\}\}$ of $[0,1]$ into $D\ge 1$ intervals of lengths $1/D$. For  
\[
D=D(\alpha,M,n)=\min\ac{k\in\N,\; \left(\frac{Mn}{2\log(en)}\right)^{\frac{1}{1+\alpha}}\le k},
\]
the $\rho$-estimator $\widehat\bsg$ based on (any) countable and dense subset $\bsS$ of $\overline \bsS$ (with respect to supremum norm) satisfies 
\begin{align}
\sup_{\bsg\et\in \cH_{\alpha}(M)}\!\!\E\cro{h^{2}\pa{R_{\bsg\et},R_{\widehat \bsg}}}\le 2C'\cro{\pa{\frac{(M/2)^{1/\alpha}\log(en)}{n}}^{\frac{\alpha}{1+\alpha}}+\frac{3\log(en)}{2n}}\nonumber
\end{align}
where $C'$ is the numerical constant appearing in~\eref{eq-risk}.
\end{prop}

\section{Calculation of $\rho$-estimators and simulation study}\label{sect-ST} 
In this section, we study the performance of the $\rho$-estimator $\widehat \bst$ of the regression function $\bst\et$ in the cases of Examples~\ref{exa-L}, \ref{exa-P}, \ref{exa-E} which correspond respectively to the logit regression, Poisson and exponential distributions parametrized by their natural parameters.


\subsubsection*{The models}
The function space $\bsT$ consists of functions $\bst$ on $\sW=\R^{5}$ with values in $I$ and for $w=(w_{1},\ldots,w_{5})\in\sW$ the value $\bst(w)$ has the following form:\\
--- In the Bernoulli model, $I=\R$ and 
\begin{equation}\label{eq-logit}
\bst(w)=\eta_{0}+\sum_{j=1}^{5}\eta_{j}w_{j}\quad \text{with $\eta=(\eta_{0},\ldots,\eta_{5})\in\R^{6}$}.
\end{equation}
--- In the Poisson model, $I=\R$ and
\begin{equation}\label{eq-poisson}
\bst(w)=\log\log\cro{1+\exp\pa{\eta_{0}+\sum_{j=1}^{5}\eta_{j}w_{j}}}\quad \text{with $\eta=(\eta_{0},\ldots,\eta_{5})\in\R^{6}$}.
\end{equation}
--- In the exponential model, $I=(0,+\infty)$ and
\begin{equation}\label{eq-exp}
\bst(w)=\log\cro{1+\exp\pa{\eta_{0}+\sum_{j=1}^{5}\eta_{j}w_{j}}}\quad \text{with $\eta=(\eta_{0},\ldots,\eta_{5})\in\R^{6}$}.
\end{equation}
For all these cases, the set $\bsT$ is VC-subgraph with dimension not larger than 7. For the calculation of the estimator on a computer, we do as if $\bsT$ {were} countable and consequently take $\bsTT=\bsT$.

\subsubsection*{The competitors}
We compare the performance of $\widehat \bst$ to that of the MLE and, in cases of Examples~\ref{exa-P} and~\ref{exa-E}, to a median-based estimator $\widehat \bst_{0}$. The estimator $\widehat \bst_{0}$ is defined as any minimizer over $\bsT$ of the criterion 
\begin{equation}\label{eq-median}
\bst\mapsto \sum_{i=1}^{n}\ab{Y_{i}-m(\bst(W_{i}))}
\end{equation}
where $m(\theta)$ is the median (or an approximation of it) of the distribution $Q_{\theta}$ for $\theta\in I$. We take 
$m(\theta)=e^{\theta}+1/3-0.02e^{-\theta}$ for the Poisson distribution with parameter $e^{\theta}$ and $m(\theta)=(\log 2)/\theta$ for the exponential one with parameter $\theta$.

In the examples we have chosen, the log-likelihood function is concave with respect to the parameter $\eta\in\R^{6}$ and the MLE is calculated using the stats4 R-package. The criterion~\eref{eq-median}
 is not convex with respect to the parameter $\eta$ and the median-based estimator is calculated using the cmaes R-package based on the CMA (Covariance Matrix Adaptation) method which turns out to be more stable than the gradient descent method. For more details about the CMA method, we refer the reader to \cite{Hansen:aa}.

%
%

\subsection{Calculation of the $\rho$-estimator}
As mentioned in Section~\ref{sect-MR}, we call $\rho$-estimator $\widehat \bst=\widehat \bst(\bsX)$ any element of the random set
\[
\sE(\bsX)=\ac{\bst\in \bsTT\;\mbox{ such that }\; \gup(\bsX,\bst)\le \inf_{\bst'\in\bsTT}\gup(\bsX,\bst')+{\frac{\kappa}{25}}},
\]
where 
\[
\gup(\bsX,\bst)=\sup_{\bst'\in\bsTT}\gT(\bsX,\bst,\bst')=\sup_{\bst'\in\bsTT}\sum_{i=1}^{n}\psi\pa{\sqrt{\frac{q_{\bst'}(X_{i})}{q_{\bst}(X_{i})}}}\quad\mbox{for all }\bst\in\bsTT.
\]

{To calculate the $\rho$-estimator $\widehat \gtheta$ we use the iterative Algorithm~1 described below. We stop it either when the condition $\gup(\bsX,\widehat \bst)\le 1$ is met or otherwise after $L=100$ iterations. Since 
\[
\gup(\bsX,\bst)=\sup_{\bst'\in \bsTT}\gT(\bsX,\bst,\bst')\ge \gT(\bsX,\bst,\bst)=0, 
\]
the quantity $ \inf_{\bst\in\bsTT}\gup(\bsX,\bst)$ is nonnegative and when $\gup(\bsX,\widehat \bst)\le 1$, 
\[
\gup(\bsX,\widehat \bst)\le  \inf_{\bst\in\bsTT}\gup(\bsX,\bst)+1\le \inf_{\bst\in\bsTT}\gup(\bsX,\bst)+\frac{\kappa}{25},
\]
which shows that $\widehat \bst$ is a $\rho$-estimator. The constant 1 has nothing magical, we just believe that the closer $\gup(\bsX,\widehat \bst)$ to $\inf_{\bst\in\bsTT}\gup(\bsX,\bst)$ the better $\widehat \bst$ performs. The condition $\gup(\bsX,\widehat \bst)\le 1$ can therefore be seen as an early stopping time that guarantees that the estimator $\widehat \bst$ almost minimizes $\bst\mapsto \gup(\bsX,\bst)$ over $\bsTT$. In all our simulations, {(including the cases when we stop after 100 iterations)}, the resulting estimators $\widehat \bst$ satisfy
\[
\gup(\bsX,\widehat \bst)\le \frac{\kappa}{25}\quad \text{hence}\quad \gup(\bsX,\widehat \bst)\le \inf_{\bst\in\bsTT}\gup(\bsX,\bst)+\frac{\kappa}{25}
\]
and are therefore $\rho$-estimators.
}

{The algorithm is based on the following heuristic that we describe in the situation where the data are  i.i.d.\ with distribution $P\et=P_{\bst\et}$ for the sake of simplicity. It is proven  in~\cite{BY-TEST}[Proposition~14], that $\gT(\bsX,\bst_{0},\bst_{1})$ is a good test statistic for testing $\cH_{0}:$ ``$P_{\bst_{0}}$ is closer (in Hellinger distance) to $P\et$ than $P_{\bst_{1}}$'' against $\cH_{1}:$ ``$P_{\bst_{1}}$ is closer to $P\et$ than $P_{\bst_{0}}$''. 
More precisely, with a probability close to 1, $\gT(\bsX,\bst_{0},\bst_{1})>0$ when $h^{2}(P_{\bst_{1}},P\et)\ll h^{2}(P_{\bst_{0}},P\et)$ and $1/n\ll h^{2}(P_{\bst_{0}},P\et)$ while the test statistic $\gT(\bsX,\bst_{0},\bst_{1})<0$ when $h^{2}(P_{\bst_{0}},P\et)\ll h^{2}(P_{\bst_{1}},P\et)$ and $1/n\ll h^{2}(P_{\bst_{1}},P\et)$.
 Note that if $h^{2}(P_{\bst},P\et)\approx 1/n$ for both $\bst=\bst_{0}$ and $\bst=\bst_{1}$, the two distributions $P_{\bst_{0}}$ and $P_{\bst_{1}}$ are both close to $P\et$ and choosing between $\bst_{0}$ and $\bst_{1}$ is unimportant. Because of these properties, if we start from an initial point $\bst_{0}$ that is not too far from $\bst\et$ and if $\bst_{1}$ is such that $\gT(\bsX,\bst_{0},\bst_{1})>0$, it is likely that one of the two following situations occur: 
\\
--- the quantity $h^{2}(P_{\bst_{1}},P\et)$ is smaller or at least of comparable order as $h^{2}(P_{\bst_{0}},P\et)$;\\
--- $h^{2}(P_{\bst_{1}},P\et)$ is of order  $1/n$.\\
In any case,  either $\bst_{1}$ improves on $\bst_{0}$ or, at least, performs similarly. We can then repeat the test starting now from $\bst_{1}$ and looking from some $\bst_{2}$ such that $\gT(\bsX,\bst_{1},\bst_{2})>0$ and so on. }

Since $\rho$-estimators are not unique, there is no reason for the algorithm to converge to a point and we are not expecting the algorithm to do so. Since the algorithm is based on the test statistic $\gT(\bsX,\bst,\bst')$, that provides a robust test between the probabilities $\gP_{\bst}$ and $\gP_{\bst'}$, as explained above, we expect the algorithm to get closer to the truth as we iterate it. As we shall see, only few iterations are in general necessary to meet the condition $\gup(\bsX,\widehat \bst)\le 1$ and when it is not the case, the estimator obtained after $L=100$ iterations provides a suitable estimation of the parameter. To find a maximizer of the mapping $\bst\mapsto \gT(\bsX,\widehat \bst,\bst)$ at each iteration, we use the cmaes R-package.

\begin{algorithm}[h]\label{algo1}
\renewcommand{\algorithmicrequire}{\textbf{Input:}}
\renewcommand{\algorithmicensure}{\textbf{Output:}}
\caption{Searching for the $\rho$-estimator}
\begin{algorithmic}[1]
\REQUIRE ~~\\
$\bsX=(X_{1},\cdots, X_{n})$: the data\\
$\bst_{0}$: the starting point\\
\ENSURE $\widehat \bst$
\STATE Initialize $l=0$, $\widehat\bst=\bst_{0}$;
\WHILE{$\gup(\bsX,\widehat\bst)>1$ and $l\leq L$}
\STATE $l\leftarrow l+1$
\STATE $\bst_{1}=\argmax\limits_{\bst\in\bsTT}\gT(\bsX,\widehat\bst,\bst)$
\STATE $\widehat\bst\leftarrow\bst_{1}$
\ENDWHILE
\STATE Return $\widehat\bst$.
\end{algorithmic}
\end{algorithm}

To intialize the process we choose the value of $\bst_{0}$ as follows. In the case of Bernoulli regression, we take for $\bst_{0}$ the function on $\R^{d}$ that  minimizes on $\bsT$ the penalized criterion (that can be found in the e1071 R-package) 
\[
\bst\mapsto 10\sum_{i=1}^{n}\pa{1-(2Y_{i}-1)\bst(W_{i})}_{+}+\frac{1}{2}\sum_{i=1}^{d}\ab{\bst(e_{i})-\bst(0)}^{2},
\]
where $e_{1},\ldots,e_{d}$ denotes the canonical basis of $\R^{d}$ (with $d=6$). The e1071 R-package is used for the purpose of classifying the $Y_{i}$ from the $W_{i}$. For the other exponential families we choose for $\bst_{0}$ the median-based estimator $\widehat \bst_{0}$. 


\subsection{Comparisons of the estimators when the model is exact}\label{sect-comp1}
Throughout this section, we assume that the data $X_{1},\ldots,X_{n}$ are i.i.d.\ with distribution $P_{\bst\et}=Q_{\bst\et}\cdot P_{W}$, $\bst\et\in \bsT$, and we estimate the risk 
\[
R_{n}(\widetilde \bst)=\E\cro{h^{2}\pa{P_{\bst\et},P_{\widetilde \bst}}}=\E\cro{\int_{\sW}h^{2}\pa{Q_{\bst\et(w)},Q_{\widetilde \bst(w)}}dP_{W}(w)}
\]
of an estimator $\widetilde \bst(\bsX)$ by the Monte Carlo method on the basis of 500 replications. For this simulation study $n=500$. We recall that, for a natural exponential family, 
\begin{equation}\label{eq-CalculHellinger}
h^{2}(Q_{\theta},Q_{\theta'})=1-\exp\cro{A\pa{\frac{\theta+\theta'}{2}}-\frac{A(\theta)+A(\theta')}{2}}
\end{equation}
where $A$ is given in (\ref{MEb}).


\paragraph{\bf Bernoulli model.}
We consider the function $\bst\et=\bst$ given by~\eref{eq-logit} with $\eta=(1,\ldots,1)\in\R^{6}$.  The distribution $P_{W}$ is $(P_{W}^{(1)}+P_{W}^{(2)}+P_{W}^{(3)})/3$ where 
$P_{W}^{(1)},P_{W}^{(2)}$ and $P_{W}^{(3)}$ are respectively the uniform distributions on the cubes 
\[
[-a,a]^{5},\quad  \cro{b-0.25,b+0.25}^{5}\quad \text{and}\quad \cro{-b-0.25,-b+0.25}^{5}
\]
with $a=0.25$ and $b=2$. 

\paragraph{\bf Poisson model}
In this case $\bst\et=\bst$ given by~\eref{eq-poisson} with $\eta=(0.7,3,4,10,2,5)$. The distribution $P_{W}$ is $P_{W,1}^{\otimes 2}\otimes P_{W,2}\otimes P_{W,3}^{\otimes 2}$ where $P_{W,1},P_{W,2}$ and $P_{W,3}$ are the uniform distributions on $[0.2,0.25]$, $\cro{0.2,0.3}$ and $\cro{0.1,0.2}$ respectively.

\paragraph{\bf Exponential model}
We set $\bst\et=\bst$ given by~\eref{eq-exp} with $\eta=(0.07,3,4,6,2,1)$. The distribution $P_{W}$ is $P_{W,1}^{\otimes 3}\otimes P_{W,2}^{\otimes 2}$ where $P_{W,1}$ and $P_{W,2}$ are the uniform distributions on $\cro{0,0.01}
$ and $\cro{0,0.1}$ respectively. 

In order to compare the performance of the $\rho$-estimator to the two other competitors we proceed as follows: we estimate the risk $R_{n}(\widehat\bst)$ of $\widehat \bst$ by Monte Carlo as explained before. We then use this quantity as a benchmark and given another estimator $\widetilde \bst$ we compute the quantity
%
\begin{equation}\label{eq-excess}
\cE(\widetilde \bst)=\frac{R_{n}(\widetilde\bst)-R_{n}(\widehat\bst)}{R_{n}(\widehat\bst)}\quad \text{so that}\quad R_{n}(\widetilde\bst)=\pa{1+\cE(\widetilde \bst)}R_{n}(\widehat\bst).
\end{equation}
Note that large positive values of $\cE(\widetilde \bst)$ indicate a significant superiority of our estimator as compared to $\widetilde \bst$ and negative values inferiority.
The respective values of $R_{n}(\widehat\bst)$ and $\cE(\widetilde \bst)$ are displayed in Table~\ref{T-L1b}. The computation time of each estimator is displayed in Table~\ref{CT-WS}.
\begin{table}[h]
\centering
\caption{Values of $R_{n}(\widehat\bst)$ and $\cE(\widetilde \bst)$ when the model is well-specified}\label{T-L1b}
\begin{tabular}{cccc}
\toprule
& $R_{n}(\widehat\bst)$& $\cE(\text{MLE})$ & $\cE(\widehat \bst_{0})$ \\
\midrule
Logit  & 0.0015& <+0.1\% & $-$ \\
\midrule
Poisson  & 0.0015&<+0.1\% & +450\%\\
\midrule
Exponential  & 0.0015& <+0.1\% & +110\%\\
\bottomrule
\end{tabular}  
\end{table}

\begin{table}[h]
\centering
\caption{Average computation time when the model is well-specified}\label{CT-WS}
\begin{tabular}{cccc}
\toprule
& $\rho$-estimator& MLE & Median-based\\
\midrule
Logit  & 331.43s&0.17s & $-$ \\
\midrule
Poisson  & 216.23s&0.23s & 34.69s\\
\midrule
Exponential  & 87.78s& 0.28s & 16.31s\\
\bottomrule
\end{tabular}  
\end{table}

{Since the median of the Bernoulli distribution is either 0 or 1, hence only weakly depends on the value of the parameter, there is no estimator of the regression function based on the median for the Bernoulli model.}

We observe the following facts:
\begin{itemize}
\item When the model is correct, the risks of the MLE and $\widehat \bst$ are the same (the value of $\cE(\text{MLE})$ is not larger than $1/1000$). In fact, a look at the simulations shows that the $\rho$-estimator coincides most of the time with the MLE, a fact which is consistent with the result proved in Baraud {\em et al.}~\citeyearpar{MR3595933} (Section~5) that states the following: under {suitable} (strong enough) assumptions, the MLE is a $\rho$-estimator when the statistical model is regular, exact and $n$ is large enough. Our simulations indicate that the result actually holds under weaker assumptions.
\item Both the MLE and the $\rho$-estimator outperform the median-based estimator $\widehat \bst_{0}$. 
\item The quantities $R_{n}(\widehat\bst)$ are of order $0.0015$ in all three cases. This fact can be explained as follows. In a regular statistical model $\sM_{0}=\{P_{\geta},\; \geta\in S\}$ parametrized with a parameter $\geta\in S\subset \R^{d}$, the asymptotic normality properties of the MLE $\widehat \geta_{n}$ together with the local equivalence of the Hellinger distance with the Euclidean one imply that, when the data are  i.i.d.\ with distribution $P_{\geta\et}\in\sM_{0}$, 
\[
n\E\cro{h^{2}(P_{\widehat \geta_{n}},P_{\geta\et})}\CV{n\to +\infty} \frac{d}{8}.
\]
In our simulation, conditionally to $W$, the distribution of $Y$ is given by an exponential family parametrized by $d=6$ parameters and the number of data being $n=500$, we expect a risk of order $d/(8n)=0.0015$, which is exactly what we obtained.  
\item The above result provides evidence that the algorithm we use does calculate the $\rho$-estimator as expected. 
\item In all the simulations we carried out, the algorithm required at most {\bf two iterations} before the stopping condition $\gup(\bsX,\widehat \bst)\le 1$ was met. 
\end{itemize}

{In the Bernoulli model, we also consider the case where the true regression function $\bst\et=\bst$ is given by~\eref{eq-logit} with $\eta=(1,\ldots,1)\in\R^{6}$ and $P_{W}=(P_{W}^{(2)}+P_{W}^{(3)})/2$. In such a situation, the MLE is likely not to exist because the sets of data labelled by $1$ and $0$ respectively can be perfectly separated by a hyperplane with probability close to 1. As expected the stats4 R-package  for calculating the MLE returns an error. In {contrast}, the $\rho$-estimator always exists and its estimated risk $R_{n}(\widehat\bst)$ is of order 0.000179. In the 100 simulations we carried out, the algorithm stops after at most 2 iterations.
}


\subsection{Comparisons of the estimators in presence of outliers}
We now work with $n=501$ independent random variables $\etc{X}$. The $500$ first variables $X_{1},\ldots,X_{n-1}$ are i.i.d.\ with distribution $P_{\bst\et}$ and simply follow the framework of the previous section. The last observation is chosen as follows. In the Bernoulli model $W_{n}=1000(1,1,1,1,1)$ and $Y_{n}=-1$, for the Poisson case $W_{n}=0.1(1,1,1,1,1)$ and $Y_{n}=200$ and for the exponential case $W_{n}=5\times 10^{-3}(1,1,1,10,10)$ and $Y_{n}=1000$. The results are displayed in Table~\ref{T-L3b} on the basis of 500 replications. The computation times for each estimator are given in Table~\ref{CT-OO}.

%
\begin{table}[h]\
\centering
\caption{Values of $R_{n}(\widehat\bst)$ and $\cE(\widetilde \bst)$ in presence of an outlier}\label{T-L3b}
\begin{tabular}{cccc}
\toprule
& $R_{n}(\widehat\bst)$& $\cE(\text{MLE})$ & $\cE(\widehat \bst_{0})$ \\
\midrule
Logit  & 0.0015& +13000\% & $-$\\
\midrule
Poisson  & 0.0019& +1900\%& +330\%\\
\midrule
Exponential  & 0.0018& +6000\%& +78\%\\
\bottomrule
\end{tabular}  
\end{table}
\begin{table}[h]
\centering
\caption{Average computation time in presence of an outlier}\label{CT-OO}
\begin{tabular}{cccc}
\toprule
& $\rho$-estimator& MLE & Median-based\\
\midrule
Logit  & 497.31s&0.12s & $-$ \\
\midrule
Poisson  & 229.36s&0.29s & 35.83s\\
\midrule
Exponential  & 103.05s& 0.32s & 15.18s\\
\bottomrule
\end{tabular}  
\end{table}

We observe the following facts:
\begin{itemize}
\item the risks of the $\rho$-estimator are quite similar to those given in Table~\ref{T-L1b} despite the presence of an outlier among the data set;
\item the MLE behaves poorly;
\item the performance of $\widehat \bst$ remains much better than that of the median-based estimator $\widehat \bst_{0}$.
\end{itemize}

Let us now display the quartiles of the distribution of the number of iterations that have been necessary to compute the $\rho$-estimator.
\begin{table}[h]
\centering
\caption{Quartiles for the number of iterations in presence of outliers}\label{T-L4b}
\begin{tabular}{ccccc}
\toprule
& 1st Quartile & Median & 3rd Quartile & Maximum\\
\midrule
Logit & 3 & 3 & 3 & 6\\
\midrule
Poisson &  2 & 2 & 2 & 3\\
\midrule
Exponential  & 2 & 2 & 2 & 3\\
\bottomrule
\end{tabular} 
\end{table}

Table \ref{T-L4b} shows that the computation of the $\rho$-estimator requires only a few iterations of the algorithm.


\subsection{Comparisons of the estimators when the data are contaminated}
We now set $n=500$ and define $P_{\bst\et}$ and $P_{W}$ as in Section~\ref{sect-comp1}. We now assume that $X_{1},\ldots,X_{n}$ are i.i.d.\ with distribution  $P\et=0.95P_{\bst\et}+0.05R$ for some (contaminating) distribution $R$ on $\sW\times \sY$ with first marginal given by $P_{W}$. We restrict ourselves to the Poisson and exponential cases (we exclude the Bernoulli model since the Bernoulli distribution remains stable under the contamination by another Bernoulli distribution). In the Poisson case, we choose for $R$ the distribution of the random variable $(W,80+B)$ where the conditional distribution of $B$ given $W=(w_{1},\ldots,w_{5})$ is Bernoulli with mean $(1+\exp\cro{-\pa{w_{1}-w_{2}-w_{4}+w_{5}}})^{-1}$. In the case of the exponential distribution $R=P_{W}\otimes \cU([50,60])$ where $\cU([50,60])$ denotes the uniform distribution on $[50,60]$. 

We measure the performance of an estimator $\widetilde \bst$ of $\bst\et$ by means of the quantity
\[
\overline R_{n}(\widetilde \bst)=\E\cro{h^{2}\pa{P\et,P_{\widetilde \bst}}}
\]
that we evaluate by Monte Carlo on the basis of 500 replications. We compare the performance of $\widehat \bst$ to a competitor $\widetilde \bst$ by evaluating the quantity
%
\begin{equation}\label{eq-excess}
\overline\cE(\widetilde \bst)=\frac{\overline R_{n}(\widetilde\bst)-\overline R_{n}(\widehat\bst)}{\overline R_{n}(\widehat\bst)}.
\end{equation}

The results are displayed in Table~\ref{T-L5b} and the computation times in Table~\ref{CT-MC}.

\begin{table}[h]\
\centering
\caption{Values of $\overline R_{n}(\widehat\bst)$ and $\overline \cE(\widetilde \bst)$ under contamination (5\%)}\label{T-L5b}
\begin{tabular}{cccc}
\toprule
& $\overline R_{n}(\widehat\bst)$& $\overline \cE(\text{MLE})$ & $\overline \cE(\widehat\bst_{0})$ \\
\midrule
Poisson & 0.028& +760\%& +11\%\\
\midrule
Exponential & 0.040& +320\%& --17\%\\
\bottomrule
\end{tabular}  
\end{table}
\begin{table}[h]
\centering
\caption{Average computation time under contamination (5\%)}\label{CT-MC}
\begin{tabular}{cccc}
\toprule
& $\rho$-estimator& MLE & Median-based\\
\midrule
Poisson  & 867.50s&0.33s & 39.23s\\
\midrule
Exponential  & 1863.97s& 0.30s & 20.47s\\
\bottomrule
\end{tabular}  
\end{table}

Let us now comment these results.
\begin{itemize}
\item With our choices of the contaminating distributions $R$, the (squared) Hellinger distance between the true distribution $P\et$ of the data and the model  is of order $h^{2}\pa{P\et,P_{\bst\et}}\approx 0.025 $. As expected, we get that $\overline R_{n}(\widehat\bst)\ge 0.025\approx h^{2}\pa{P\et,P_{\bst\et}}$. Note that the situation is extreme in the sense that the approximation error is much larger than estimation error  that can be achieved when the model is well specified (which is about 0.0015). This means that the model is ``very'' misspecified. 

\item The MLE behaves poorly.
\item In the exponential case, the median-based estimator $\widehat \bst_{0}$ outperforms the $\rho$-estimator while the opposite situation occurs in the Poisson case.
\end{itemize}
\begin{table}[h]
\centering
\caption{Quartiles for the number of iterations when the data are contaminated}\label{T-L6b}
\begin{tabular}{ccccc}
\toprule
& 1st Quartile & Median & 3rd Quartile & Maximum\\
\midrule
Poisson  &  5 & 5 & 5 & 100\\
\midrule
Exponential  & 5 & 10 & 30 & 100\\
\bottomrule
\end{tabular} 
\end{table}

In Table~\ref{T-L6b}, we observe that the number of iterations for calculating the $\rho$-estimator increases substantially as compared to the two previous situations.  We note that for some simulations the algorithm was iterated 100 times (which corresponds to the maximal number of iterations that we allow) and the stopping condition  $\gup(\bsX,\widehat \bst)\le 1$ was not met (but satisfies $\gup(\bsX,\widehat \bst)\le\kappa/25$). Despite this fact, the estimator that we get at the final {step}, hence after 100 iterations, performs well since the values of the risks $\overline R_{n}(\widehat\bst)$ are of the same order as $h^{2}\pa{P\et,P_{\bst\et}}$ and comparable to the median-based estimator $\widehat \bst_{0}$.

\section{An upper bound on the expectation of the supremum of an empirical process over a VC-subgraph class}\label{sect-EmpProc}
The aim of this section is to prove the following result.

\begin{thm}\label{thm-VCB}
Let $X_{1},...,X_{n}$ be $n$ independent random variables with values in $(\sX,\cX)$ and $\sF$ an at most countable VC-subgraph class of functions with values in $[-1,1]$ and VC-dimension not larger than $V\ge 1$. If 
\[
Z(\mathscr{F})=\sup \limits_{f\in\mathscr{F}} \left|\sum \limits_{i=1}^{n}(f(X_{i})-\E\cro{f(X_{i})})\right|\;\; \text{and}\;\; \sup_{f\in\sF}\frac{1}{n}\sum_{i=1}^{n}\E\cro{f^{2}(X_{i})}\le \sigma^{2}\le 1,
\]
then
\begin{equation}\label{bound-empirical-process}
\E\cro{Z(\mathscr{F})}\le 4.74\sqrt{nV\sigma^{2}\sL(\sigma)}+90V\sL(\sigma),
\end{equation}
with $\sL(\sigma)=9.11+\log(1/\sigma^{2})$.
\end{thm}
Let us now turn to the proof. 
It follows from classical symmetrisation arguments that $\E\cro{Z(\mathscr{F})}\le 2\E\cro{\overline Z(\mathscr{F})}$, where $\overline{Z}(\mathscr{F})=\sup \limits_{f\in\mathscr{F}}\left|\sum \limits_{i=1}^{n}\varepsilon_{i}f(X_{i})\right|$ 
%
%
and $\eps_{1},\ldots,\eps_{n}$ are i.i.d.\ Rademacher random variables. It is therefore enough to prove that 
\begin{equation}\label{b-emp-process2}
\E\cro{\overline{Z}(\mathscr{F})}\le 2.37\sqrt{nV\sigma^{2}\sL(\sigma)}+45V\sL(\sigma).
\end{equation}
Given a probability $P$ and a class of functions $\sG$ on  $(E,\cE)$ we denote by $N_{r}(\epsilon,\sG,P)$ the smallest cardinality of an $\epsilon$-net for the $\L_{r}(E,\cE,P)$-norm $\norm{\cdot}_{r,P}$, i.e.\ the minimal cardinality of a subset $\sG[\epsilon]$ of $\sG$ that satisfies for all $g\in\sG$ 
\[
\inf_{\overline g\in\sG[\epsilon]}\norm{g-\overline g}_{r,P}=\inf_{\overline g\in\sG[\epsilon]}\pa{\int_{E}\ab{g-\overline g}^{r}dP}^{1/r}\le \epsilon.
\]
We start with the following lemma. 
\begin{lem}\label{entropy-bound} 
Whatever the probability $P$ on $(\sX,\cX)$, $\epsilon\in (0,2)$ and $r\ge 1$ 
\[
N_{r}(\epsilon, \mathscr{F}, P)\le e (V+1)(2e)^{V}\pa{\frac{2}{\epsilon}}^{rV}.
\]
\end{lem}

\begin{proof}[Proof of Lemma~\ref{entropy-bound}]
Let $\lambda$ be the Lebesgue measure on $([-1,1],\sB([-1,1]))$ and $Q$ the product probability $P\otimes (\lambda/2)$ on $(E,\cE)=(\sX\times[-1,1],\cX\times \sB([-1,1]))$. Given two elements $f,g\in\sF$ and $x\in \sX$ 
\begin{align*}
\int_{[-1,1]}\ab{\1_{f(x)>t}-\1_{g(x)>t}}dt&=\int_{[-1,1]}\pa{\1_{f(x)>t\ge g(x)}+\1_{g(x)>t\ge f(x)}}dt\\
&=\ab{f(x)-g(x)}
\end{align*}
and, setting $C_{f}=\{(x,t)\in\sX\times [-1,1],\; f(x)>t\}$ the subgraph of $f$ and similarly $C_{g}$ that of $g$, we deduce from Fubini's theorem that
\begin{align*}
\norm{f-g}_{1,P}&=\int_{\sX}\ab{f-g}dP=2\int_{\sX\times [-1,1]}\ab{\1_{C_{f}}(x,t)-\1_{C_{g}}(x,t)}dQ\\
&=2\norm{\1_{C_{f}}-\1_{C_{g}}}_{1,Q}.
\end{align*}
Since the functions $f,g\in\sF$ take their values in $[-1,1]$, 
\[
\norm{f-g}_{r,P}^{r}=\int_{\sX}\ab{f-g}^{r}dP\le 2^{r-1}\int_{\sX}\ab{f-g}dP\le 2^{r}\norm{\1_{C_{f}}-\1_{C_{g}}}_{1,Q}
\]
and consequently, for all $\epsilon>0$ 
\[
N_{r}(\epsilon,\sF,P)\le N_{1}((\epsilon/2)^{r},\sG,Q)\quad \text{with}\quad \sG=\{\1_{C_{f}},\; f\in\sF\}.
\]
Since $\sF$ is VC-subgraph with VC-dimension not larger than $V$, the class $\sG$ is by definition VC with dimension not larger than $V$ and the result follows from Corollary~1 in Haussler~\citeyearpar{MR1313896}. 
\end{proof}

The proof of Theorem~\ref{thm-VCB} is based on a chaining argument. It follows from the  monotone convergence theorem that it is actually enough to prove~\eref{b-emp-process2} with $\sF_{J}$, $J\ge 1$, in place of $\sF$ where $(\sF_{J})_{J\ge 1}$ is a sequence of finite subsets of $\sF$ which is increasing for the inclusion and satisfies $\bigcup_{J\ge 1}\sF_{J}=\sF$. We may therefore assume with no loss of generality that $\sF$ is finite. 

Let $q$ be some positive number in $(0,1)$ to be chosen later on and $P_{\bsX}$ the empirical distribution $n^{-1}\sum_{i=1}^{n}\delta_{X_{i}}$. We shall denote by $\E_{\eps}$ the expectation with respect to the Rademacher random variables $\eps_{i}$, hence conditionally on $\bsX=(X_{1},\ldots,X_{n})$. Let $\norm{\cdot}_{2,\bsX}$ be the $\L_{2}(\sX,\cX,P_{\bsX})$-norm and 
\[
\widehat \sigma^{2}=\widehat \sigma^{2}(\bsX)=\sup_{f\in\sF}\norm{f}_{2,\bsX}^{2}=\sup_{f\in\sF}\cro{\frac{1}{n}\sum_{i=1}^{n}f^{2}(X_{i})}\in [0,1].
\]
For each positive integer $k$, let $\sF_{k}=\sF_{k}(\bsX)$ be a minimal $(q^{k}\widehat \sigma)$-net for $\sF$ with respect to $\norm{\cdot}_{2,\bsX}$. In particular, we can associate to a function $f\in\sF$ a sequence $(f_{k})_{k\ge 1}$ with $f_{k}\in\sF_{k}$ satisfying $\norm{f-f_{k}}_{2,\bsX}\le q^{k}\widehat \sigma$ for all $k\ge 1$. Actually, since $\sF$ is finite $f_{k}=f$ for all $k$ large enough. Besides, it follows from Lemma~\ref{entropy-bound} with the choices $r=2$ and $P=P_{\bsX}$ that for all $k\ge 1$ we can choose $\sF_{k}$ in such a way that $\log[\Card\sF_{k }]$ is not larger than $h(q^{k}\widehat \sigma)$ where 
\begin{equation}\label{eq-entropy}
h(\epsilon)=\log\cro{e(V+1)(2e)^{V}}+2V\log\pa{\frac{2}{\epsilon}}\quad \text{for all $\epsilon\in (0,1]$.}
\end{equation}
For $f\in \sF$, the following (finite) decomposition holds 
\begin{align*}
\sum_{i=1}^{n}\eps_{i}f(X_{i})&=\sum_{i=1}^{n}\eps_{i}f_{1}(X_{i})+ \sum_{i=1}^{n}\eps_{i}\sum_{k=1}^{+\infty}\cro{f_{k+1}(X_{i})-f_{k}(X_{i})}\\
&=\sum_{i=1}^{n}\eps_{i}f_{1}(X_{i})+\sum_{k=1}^{+\infty}\cro{\sum_{i=1}^{n}\eps_{i}\pa{f_{k+1}(X_{i})-f_{k}(X_{i})}}.
\end{align*}
Setting $\sF_{k}^{2}=\{(f_{k},f_{k+1}),\; f\in\sF\}$ for all $k\ge 1$, we deduce that  
\[
\overline Z(\sF)\le \sup_{f\in\sF_{1}}\ab{\sum_{i=1}^{n}\eps_{i}f(X_{i})}+\sum_{k=1}^{+\infty}
\sup_{(f_{k},f_{k+1})\in \sF_{k}^{2}}\ab{\sum_{i=1}^{n}\eps_{i}\cro{f_{k+1}(X_{i})-f_{k}(X_{i})}}
\]
and consequently, 
\begin{align*}
\E_{\eps}\cro{\overline Z(\sF)}&\le \E_{\eps}\cro{\sup_{f\in\sF_{1}}\ab{\sum_{i=1}^{n}\eps_{i}f(X_{i})}}\\
&\quad +  \sum_{k=1}^{+\infty}\E_{\eps}\cro{\sup_{(f_{k},f_{k+1})\in\sF_{k}^{2}}\ab{\sum_{i=1}^{n}\eps_{i}\cro{f_{k}(X_{i})-f_{k+1}(X_{i})}}}.
\end{align*}

Given a finite set $\sG$ of functions on $\sX$ and setting $-\sG=\{-g,\; g\in \sG\}$ and $v^{2}=\max_{g\in \sG}\|g\|_{2,\bsX}^{2}$, 
we shall repeatedly use  the inequality 
\[
\E\cro{\sup \limits_{g\in \sG}\left|\sum \limits_{i=1}^{n}\varepsilon_{i}g(X_{i})\right|}=\E\cro{\sup \limits_{g\in \sG\cup (-\sG)}\sum \limits_{i=1}^{n}\varepsilon_{i}g(X_{i})}\le\sqrt{2n\log(2\Card \sG)v^{2}}
\]
that can be found in Massart~\citeyearpar{MR2319879}[inequality (6.3)]. Since $\max_{f\in\sF_{1}}\norm{f}_{2,\bsX}^{2}\le \widehat \sigma^{2}$, $\log( \Card \sF_{1})\le h(q\widehat \sigma)$, $\log( \Card \sF_{k}^{2})\le h(q^{k}\widehat \sigma)+h(q^{k+1}\widehat \sigma)$ and 
\begin{align*}
\lefteqn{\sup_{(f_{k},f_{k+1})\in\sF_{k}^{2}}\norm{f_{k}-f_{k+1}}_{2,\bsX}^{2}}\hspace{7mm}\\
&\le \sup_{f\in\sF}\sup_{(f_{k},f_{k+1})\in\sF_{k}^{2}}\pa{\norm{f-f_{k}}_{2,\bsX}+\norm{f-f_{k+1}}_{2,\bsX}}^{2}\le\pa{1+q}^{2} q^{2k}\widehat\sigma^{2},
\end{align*}
we deduce that 
\begin{align*}
&\E_{\eps}\cro{\sup_{f\in\sF_{1}}\ab{\sum_{i=1}^{n}\eps_{i}f(X_{i})}}\le \widehat \sigma\sqrt{2n\pa{\log 2+h(q\widehat \sigma)}},
\end{align*}
and for all $k\ge 1$
\begin{align*}
\lefteqn{\E_{\eps}\cro{\sup_{(f,g)\in \sF_{k}^{2}}\ab{\sum_{i=1}^{n}\eps_{i}\cro{g(X_{i})-f(X_{i})}}}}\hspace{30mm}\\
&\le \widehat \sigma (1+q)q^{k}\sqrt{2n\pa{\log 2+h(q^{k}\widehat \sigma)+h(q^{k+1}\widehat \sigma)}}.
\end{align*}
Setting $g:u\mapsto \sqrt{\log 2+h(u)+h(qu)}$ on $(0,1]$ and using the fact that $g$ is decreasing (since $h$ is)  we deduce that
{
\begin{align*}
\lefteqn{\E_{\eps}\cro{\overline Z(\sF)}}\quad\\
&\le \widehat \sigma\sqrt{2n}\cro{\sqrt{\log 2+h(q\widehat \sigma)}+(1+q)\sum_{k\ge 1}q^{k}\sqrt{\log 2+h(q^{k}\widehat \sigma)+h(q^{k+1}\widehat \sigma)}}\\
&\le \widehat \sigma\sqrt{2n}\cro{g(\widehat \sigma)+(1+q)\sum_{k\ge 1}q^{k}g(q^{k}\widehat \sigma)}\\
&\le \sqrt{2n}\cro{\frac{1}{1-q}\int_{q\widehat \sigma}^{\widehat \sigma}g(u)du+\frac{1+q}{1-q}\sum_{k\ge 1}\int_{q^{k+1}\widehat \sigma}^{q^{k}\widehat \sigma}g(u)du}\\
&\le \sqrt{2n}\frac{1+q}{1-q}\int_{0}^{\widehat \sigma}g(u)du.
\end{align*}
}
The mapping $g$ being positive and decreasing, the function $G:y\mapsto  \int_{0}^{y}g(u)du$ is increasing and concave. Taking the expectation with respect to $\bsX$ on both sides of the previous inequality and using Jensen's inequality we get  
\begin{align}
\E\cro{\overline Z(\sF)}&\le \sqrt{2n}\frac{1+q}{1-q}\E\cro{G(\widehat \sigma)}\le  \sqrt{2n}\frac{1+q}{1-q}G\pa{\E\cro{\widehat \sigma}}\nonumber\\
&\le \sqrt{2n}\frac{1+q}{1-q}G\pa{\sqrt{\E\cro{\widehat \sigma^{2}}}}.\label{bEZ-1}
\end{align}
By symmetrization and contraction arguments (see Theorem 4.12 in Ledoux and Talagrand~\citeyearpar{MR1102015}), 
\begin{align}
\E\cro{n\widehat \sigma^{2}}&\le\E\cro{\sup \limits_{f\in\mathscr{F}}\sum \limits_{i=1}^{n}\left(f^{2}(X_{i})-\E\cro{f^{2}(X_{i})}\right)}+\sup \limits_{f\in\mathscr{F}}\sum \limits_{i=1}^{n}\E\cro{f^{2}(X_{i})}\nonumber \\
&\le2\E\cro{\sup \limits_{f\in\mathscr{F}}\left|\sum \limits_{i=1}^{n}\varepsilon_{i}f^{2}(X_{i})\right|}+n\sigma^{2}\nonumber \\
&\le8\E\cro{\sup \limits_{f\in\mathscr{F}}\left|\sum \limits_{i=1}^{n}\varepsilon_{i}f(X_{i})\right|}+n\sigma^{2}=8\E\cro{\overline{Z}(\mathscr{F})}+n\sigma^{2}\label{inequality}
\end{align}
and we infer from~\eref{bEZ-1} that
\begin{equation}\label{bEZ-2}
\E\cro{\overline Z(\sF)}\le \sqrt{2n}\frac{1+q}{1-q}G\pa{B}\quad \text{with}\quad B=\sqrt{
\sigma^{2}+\frac{8\E\cro{\overline{Z}(\mathscr{F})}}{n}}\wedge 1.
\end{equation}

The following lemma provides an evaluation of $G$. 
\begin{lem}\label{lem-2}
Let $a,b,y_{0}$ be positive numbers and $y\in [y_{0},1]$, 
\[
\int_{0}^{y}\sqrt{a+b\log(1/u)du}\le \pa{1+\frac{b}{2a}}y\sqrt{a+b\log(1/y_{0})}.
\]
\end{lem}

\begin{proof}
Using an integration by parts and the fact that
\[
\frac{d}{du}\sqrt{a+b\log(1/u)}=-\frac{b}{2u\sqrt{a+b\log(1/u)}}
\]
we get
\begin{align*}
\int_{0}^{y}\sqrt{a+b\log(1/u)}du&=\cro{u\sqrt{a+b\log(1/u)}}_{0}^{y}+\frac{1}{2}\int_{0}^{y}\frac{b}{\sqrt{a+b\log(1/u)}}du\\
&\le y\sqrt{a+b\log(1/y)}+\frac{by}{2\sqrt{a+b\log(1/y)}}\\
&=y\sqrt{a+b\log(1/y)}\cro{1+\frac{b}{2\pa{a+b\log(1/y)}}}
\end{align*}
and the conclusion follows from the fact that $y_{0}\le y\le 1$.
\end{proof}

Since for all $y\in (0,1]$, $g(y)=\sqrt{a+b\log(1/y)}$ with 
\[
a=\log[2e^{2}(V+1)^{2}]+2V\log(8e/q)\quad \text{and }\quad b=4V
\]
we may apply Lemma~\ref{lem-2} with $y_{0}=\sigma$ and $y=B$ and deduce from~\eref{bEZ-2} that
\begin{align*}
\E\cro{\overline Z(\sF)}&\le \sqrt{2n}\frac{1+q}{1-q}\pa{1+\frac{b}{2a}}B\sqrt{a+b\log(1/\sigma)}\\
&\le \sqrt{2n}\frac{1+q}{1-q}\pa{1+\frac{b}{2a}}\sqrt{\sigma^{2}+\frac{8\E\cro{\overline Z(\sF)}}{n}}\sqrt{a+b\log(1/\sigma)}.
\end{align*}
Solving the inequality $\E\cro{\overline Z(\sF)}\le A\sqrt{2n\sigma^{2}+16\E\cro{\overline Z(\sF)}}$ with
\[
A=\frac{1+q}{1-q}\pa{1+\frac{b}{2a}}\sqrt{a+b\log(1/\sigma)},
\]
we get that 
\begin{equation}
\E\cro{\overline Z(\sF)}\le 8A^{2}+\sqrt{64A^{4}+2A^{2}n\sigma^{2}}\le 16A^{2}+A\sqrt{2n\sigma^{2}}.\label{bEZ-3}
\end{equation}
Finally, we conclude by using the inequalities 
\begin{align*}
\frac{b}{2a}&=\frac{4V}{2\cro{\log[2e^{2}(V+1)^{2}]+2V\log(8e/q)}}\le \frac{1}{\log(8e/q)},\\
\frac{a}{b}&=\frac{\log[2e^{2}(V+1)^{2}]+2V\log(8e/q)}{4V}\\
&=\frac{\log(8e/q)}{2}+\frac{\log[2e^{2}(V+1)^{2}]}{4V}\le \frac{\log(8e/q)}{2}+\frac{\log[8e^{2}]}{4}\\
&=\log\pa{\frac{8^{3/4}e}{\sqrt{q}}}
\end{align*}
which, with our choice $q=0.0185$, give
\begin{align*}
A&\le \frac{1+q}{1-q}\pa{1+\frac{1}{\log(8e/q)}}\sqrt{4V\pa{\log\pa{\frac{8^{3/4}e}{\sqrt{q}}}+\log\frac{1}{\sigma}}}\\
&\le 2.37\sqrt{V\pa{4.555+\log\frac{1}{\sigma}}}
\end{align*}
and together with~\eref{bEZ-3} {leads} to~\eref{b-emp-process2}.

\section{Proofs}\label{sect-proof}
\subsection{Proof of Theorem~\ref{thm-1}}
We recall that the function $\psi$ defined by~\eref{def-psi} satisfies Assumption~2 of \cite{BarBir2018} with $a_{0}=4$, $a_{1}=3/8$ and $a_{2}^{2}=3\sqrt{2}$ (see their Proposition~3). Theorem~\ref{thm-1} is {actually} a consequence of Theorem~1 of \cite{BarBir2018}. Set $\gmu=\bigotimes_{i=1}^{n}\mu_{i}$ with $\mu_{i}=P_{W_{i}}\otimes \nu$ for all $i\in\{1,\ldots,n\}$, denote  by $\cbP$ the following families of densities (with respect to $\gmu$) on $\sX^{n}=(\sW\times \sY)^{n}$  
\[
\cbP=\{\gp_{\bst}:\gx=(x_{1},\ldots,x_{n})\mapsto q_{\bst}(x_{1})\ldots q_{\bst}(x_{n}),\; \bst\in\bsTT\}
\]
and by $\sbP$ the corresponding  $\rho$-model, i.e.\ the countable set $\{\gP=\gp_{\bst}\cdot\gmu,\;  \bst\in\bsTT\}$ with representation $(\gmu,\cbP)$.

Let us first prove 
\begin{prop}\label{prop-cle}
Under Assumption~\ref{Hypo}, the class of functions $\cP=\{q_{\bst}:(w,y)\mapsto q_{\bst(w)}(y),\; \bst\in\bsT\}$ on $\sX=\sW\times\sY$ is VC-subgraph with dimension not larger than $9.41V$.
\end{prop}
\begin{proof}
 The exponential function being monotone, it suffices to prove that the family 
\[
\sF=\ac{f:(w,y)\mapsto S(y)\bst(w)-A(\bst(w)),\ \bst\in \bsT}
\]
is VC-subgraph on $\sX=\sW\times\sY$ with dimension not larger than $9.41V$. The function $A$ being convex and continuous on $I$, the mapping defined  on $I$ by {$\theta\mapsto S(y)\theta-A(\theta)$} is continuous and concave for all fixed $y\in\sY$. In particular, for $u\in\R$ the level set {$\{\theta\in I,\; S(y)\theta-A(\theta) > u\}$} is an open subinterval of $I$ of the form $(\underline a(y,u),\overline a(y,u))$ where $\underline a(y,u)$ and $\overline a(y,u)$ belong to the closure $\overline I$ of $I$ in $\overline \R=\R\cup\{\pm\infty\}$.  For $\bst\in \bsT$, let us set
\begin{eqnarray*}
C_{\bst}^{+}&=&\{(w,b,b')\in\sW\times \overline I^{2},\ \bst(w)>b\}\\
C_{\bst}^{-}&=&\{(w,b,b')\in\sW\times \overline I^{2},\ \bst(w)<b'\}
\end{eqnarray*}
and define $\sC^{+}$ (respectively $\sC^{-}$)  as the class of all subsets $C_{\bst}^{+}$ (respectively $C_{\bst}^{-}$) when $\bst$ varies among $\bsT$. 

Let us prove that $\sC^{+}$ is a VC-class of sets  on $\sZ=\sW\times \overline I^{2}$ with dimension not larger than $V$. If $\sC^{+}$ shatters the finite subset $\{\Etc{z}{k}\}$ of $\sZ$ with $z_{i}=(w_{i},b_{i},b'_{i})$ for $i\in\{1,\ldots,k\}$, necessarily the $b_{i}$ belong to $\R$ for all $i\in\{1,\ldots,k\}$. 
Consequently, the class of subgraphs
\[
\widetilde \sC^{+}=\ac{\st\{(w,b)\in\sW\times \R,\ \bst(w)>b\},\ \bst\in\bsT}
\]
shatters the points $\widetilde z_{1}=(w_{1},b_{1}),\ldots,\widetilde z_{k}=(w_{k},b_{k})$ in $\sW\times \R$. This is possible only for $k\le V$ since, by Assumption~\ref{Hypo}, $\bsT$ is VC-subgraph on $\sW$ with dimension $V$. 

Arguing similarly we obtain that $\sC^{-}$ is also VC on $\sZ$ with dimension not larger than $V$. In particular, it follows from van der Vaart and Wellner~\citeyearpar{MR2797943} Theorem~1.1 that the class of subsets 
\[
\sC^{+}\bigwedge\sC^{-}=\{C^{+}\cap C^{-},\ \ C^{+}\in\sC^{+},\ C^{-}\in\sC^{-}\}
\]
is VC on $\sZ$ with dimension not larger than $9.41V$.

Let us now conclude the proof. If the class of subgraphs of $\sF$ shatter the points $(w_{1},y_{1},u_{1}),\ldots,(w_{k},y_{k},u_{k})$ in $\sW\times\sY\times\R$, this means that for all subsets $J$ of $\{1,\ldots,k\}$, there exists a function $\bst=\bst(J)\in\bsT$ such that the condition $j\in J$ is equivalent to the following ones
\[
{S(y_{j})\bst(w_{j})-A\pa{\bst(w_{j})} > u_{j}}\iff\bst(w_{j})\in (\underline a(y_{j},u_{j}),\overline a(y_{j},u_{j}))
\]
and finally equivalent to 
\[
z_{j}=(w_{j},\underline a(y_{j},u_{j}),\overline a(y_{j},u_{j}))\in C_{\bst}^{+}\cap C_{\bst}^{-}.
\]
Hence, the class 
\[
\sC=\ac{C_{\bst}^{+}\cap C_{\bst}^{-},\ \bst\in\bsT}\subset \sC^{+}\bigwedge\sC^{-}
\]
shatters $\{z_{1},\ldots,z_{k}\}$ in $\sZ$. This is possible for $k\le 9.41 V$ only and proves the fact that $\sF$ is VC-subgraph with dimension not larger than $9.41 V$. 
\end{proof}
The result below provides an upper bound on the $\rho$-dimension function $(\gP,\overline\gP)\mapsto D^{\sbP}(\gP,\overline\gP)$ of $\sbP$. The $\rho$-dimension is defined in Baraud and Birg\'e~\citeyearpar{BarBir2018}[Definition 4].
\begin{prop}\label{cor-constants}
Under Assumption~\ref{Hypo}, for all product probabilities $\gP,\overline \gP=\otimes_{i=1}^{n} \overline P_{i}$ on $(\sX^{n},\cX\on)$ with $\overline P_{i}=\overline p\cdot \mu_{i}$ for all $i\in\{1,\ldots,n\}$, 
\[
D^{\sbP}(\gP,\overline\gP)\le 10^{3}V\cro{9.11+\log_{+}\pa{\frac{n}{V}}}.
\]
\end{prop}
\begin{proof}
Given two product probabilities $\gR=\otimes_{i=1}^{n}R_{i}$ and $\gR'=\otimes_{i=1}^{n}R_{i}'$ on $(\sX^{n},\cX\on)$, we set $\gh^{2}(\gR,\gR')=\sum_{i=1}^{n}h^{2}(R_{i},R_{i}')$ and for $y>0$, 
\[
\sF_{y}=\ac{\left.\psi\pa{\sqrt{\frac{q_{\bst}}{\overline p}}}\right|\; \bst\in\bsTT, \gh^{2}(\gP,\overline \gP)+\gh^{2}(\gP,\gp_{\theta}\cdot\gmu)< y^{2}}.
\]
It follows from Proposition~\ref{prop-cle} and Baraud {\em et al.}~\citeyearpar{MR3595933}[Proposition~42] that 
$\sF_{y}$ is VC-subgraph with dimension not larger than $\overline V=9.41V$. Besides, by Proposition~3 in Baraud and Birg\'e~\citeyearpar{BarBir2018} we know that our function $\psi$ satisfies their Assumption~2 and more precisely (11) which, together with the definition of $\sF_{y}$, implies that $\sup_{f\in\sF_{y}}n^{-1}\sum_{i=1}^{n}\E\cro{f^{2}(X_{i})}\le \sigma^{2}(y)=(a_{2}^{2}y^{2}/n)\wedge 1$. Applying Theorem~\ref{thm-VCB} with $\sF=\sF_{y}$, we obtain that 
\begin{align*}
w^{\sbP}(\gP,\overline\gP,y)&=\E\cro{\sup_{f\in\cF_{y}}\ab{\sum_{i=1}^{n}f(X_{i})-\E\cro{f(X_{i})}}}\\
&\le 4.74a_{2}y\sqrt{\overline V\sL(\sigma(y))}+90\overline V\sL(\sigma(y))\\
&=14.55a_{2}y\sqrt{V\sL(\sigma(y))}+846.9V\sL(\sigma(y)).
\end{align*}
Let $D\ge a_{1}^{2}V/(16a_{2}^{4})=2^{-11}V$ to be chosen later on and $\beta=a_{1}/(4a_{2})$. For $y\ge \beta^{-1}\sqrt{D}$,
\begin{align*}
\sL(\sigma(y))&=9.11 +\log_{+}\pa{\frac{n}{a_{2}^{2}y^{2}}}\le 9.11+\log_{+}\pa{\frac{n}{a_{2}^{2}\beta^{-2}D}}\\
&= 9.11+\log_{+}\pa{\frac{a_{1}^{2}n}{16a_{2}^{4}D}}\le 9.11+\log_{+}\pa{\frac{n}{V}}=L.
\end{align*}
Hence for all $y\ge \beta^{-1}\sqrt{D}$, 
\begin{align*}
w^{\sbP}(\gP,\overline\gP,y)&\le 14.55a_{2}y\sqrt{V L}+846.9 V L\\
&=\frac{a_{1}y^{2}}{8}\cro{\frac{8\times 14.55a_{2}\sqrt{V L}}{a_{1}y}+\frac{8\times 846.9 VL}{a_{1}y^{2}}}\\
&\le \frac{a_{1}y^{2}}{8}\cro{\frac{8\times 14.55a_{2}\sqrt{V L}}{a_{1}\beta^{-1}\sqrt{D}}+\frac{8\times 846.9 VL}{a_{1}\beta^{-2}D}}\\
&=\frac{a_{1}y^{2}}{8}\cro{\frac{2\times 14.55\sqrt{V L}}{\sqrt{D}}+\frac{8\times 846.9 a_{1}VL}{16 a_{2}^{2}D}}\\
&=\frac{a_{1}y^{2}}{8}\cro{\frac{29.1\sqrt{V L}}{\sqrt{D}}+\frac{37.5VL}{D}}\le\frac{a_{1}y^{2}}{8}
\end{align*}
for $D=10^{3}VL>2^{-11}V$. The result follows from the definition of the $\rho$-dimension given in Baraud and Birg\'e~\citeyearpar{BarBir2018}[Definition 4]. 
\end{proof}

Let us now {complete} the proof of Theorem~\ref{thm-1}. It follows from Baraud and Birg\'e~\citeyearpar{BarBir2018}[Theorem~1], that the $\rho$-estimator $\widehat \gP=\gP_{\widehat \bst}$ built on the $\rho$-model $\sbP$, which coincides with that described in  Section~\ref{Sect-def-est}, satisfies for all $\overline \gP\in\sbP$, with a probability at least $1-e^{-\xi}$, 
\begin{align*}
\gh^{2}(\gP\et,\widehat\gP)\le \gamma\gh^{2}\pa{\gP\et,\sbP}+\gamma'\pa{\frac{D^{\sbP}(\gP\et,\overline \gP)}{4.7}+1.49+\xi}
\end{align*}
with 
\[
\gamma=\frac{4(a_{0}+8)}{a_{1}}+2+\frac{84}{a_{2}^{2}}<150
\quad \text{and}\quad \gamma'= \frac{4}{a_{1}}\pa{\frac{35a_{2}^{2}}{a_{1}}+74}<5014
\]
and $D^{\sbP}(\gP\et,\overline \gP)\le 10^{3}V\cro{9.11+\log_{+}(n/V)}$ by Proposition~\ref{cor-constants}.
Finally, the result follows from the facts that $\gh^{2}(\gP\et,\widehat\gP)=\gh^{2}(\gQ\et,\gQ_{\widehat \bst})$ and $\gh^{2}\pa{\gP\et,\sbP}=\gh^{2}\pa{\gQ\et,\sbQ}$.

\subsection{A preliminary result}\label{sect-PRes}
\begin{prop}\label{propert}
Let $g$ be a $1$-Lipschitz function on $\R$ supported on $[0,1]$, $N$ some positive integer and $L$ some positive number. For $\eps\in \{-1,1\}^{2^{N}}$ define the function $G_{\eps}$ as 
\begin{equation}\label{def-G}
G_{\eps}(x)=L\sum_{k=0}^{2^{N}-1}\eps_{k+1}g\pa{2^{N}x-k}\quad \text{for all $x\in [0,1]$.}
\end{equation}
Then, $G_{\eps}$ satisfies~\eref{def-Hold} with $\alpha\in (0,1]$ and $M>0$ provided that  $L\le 2^{-[(N-1)\alpha+1]}M$.
\end{prop}

\begin{proof}
For $k\in \Lambda=\{0,\ldots,2^{N}-1\}$, we set $g_{k}:x\mapsto g(2^{N}x-k)$. Since $g$ is 1-Lipschitz and supported on $[0,1]$,  the function $g_{k}$ is $2^{N}$-Lipschitz on $\R$ and supported on $I_{k}=[2^{-N}k,2^{-N}(k+1)]\subset [0,1]$ for all $k\in \Lambda$. In particular, the intersection of the supports of $g_{k}$ and $g_{k'}$ reduces to at most a singleton when $k\ne k'$. 

Let $x<y$ be two points in $[0,1]$. If there exists $k\in \Lambda$ such that $x,y\in I_{k}$, using that $0\le y-x\le 2^{-N}$ and the fact that $L2^{N\alpha}\le L2^{(N-1)\alpha+1}\le  M$, we obtain that
\begin{align*}
\ab{G_{\eps}(y)-G_{\eps}(x)}&=L\ab{g_{k}(y)-g_{k}(x)}\le L2^{N}(y-x)\\
&\le L2^{N}(y-x)^{1-\alpha}(y-x)^{\alpha}\le L2^{N\alpha}(y-x)^{\alpha}\le M(y-x)^{\alpha}.
\end{align*}
%
If $x\in I_{k}$ and $y\in I_{k'}$ with $k'\ge k+1$, 
\[
\pa{y-2^{-N}k'}+\pa{2^{-N}(k+1)-x}\le 2^{-N+1}\wedge (y-x)
\]
and since $g$ vanishes at 0 and 1,
\begin{align*}
\ab{G_{\eps}(y)-G_{\eps}(x)}&=L\ab{\eps_{k'+1}g_{k'}(y)-\eps_{k+1}g_{k}(x)}\le L\ab{g_{k'}(y)}+L\ab{g_{k}(x)}\\
&=L\ab{g_{k'}(y)-g_{k'}(2^{-N}k')}+L\ab{g_{k}(2^{-N}(k+1))-g_{k}(x)}\\
&\le L2^{N}\cro{y-2^{-N}k'+2^{-N}(k+1)-x}^{1-\alpha+\alpha}\\
&\le  L2^{N} 2^{(1-\alpha)(-N+1)}(y-x)^{\alpha}=L2^{(N-1)\alpha+1}(y-x)^{\alpha}
\end{align*}
and the conclusion follows from the fact that $L\le 2^{-[(N-1)\alpha+1]}M$. 
\end{proof}

We use the following version of Assouad's lemma.
\begin{lem}[Assouad's Lemma]\label{assouad}
Let $\cP$ be a family of probabilities on a measurable space $(\sX,\cX)$. Assume that for some integer $d\ge 1$, $\cP$ contains a subset of the form $\cC=\{P_{\eps},\; \eps\in \{-1,1\}^{d}\}$  with the following properties:
\begin{listi}
\item there exists $\eta>0$ such that for all $\eps,\eps'\in\{-1,1\}^{d}$
\[
h^{2}\pa{P_{\eps},P_{\eps'}}\ge \eta \delta(\eps,\eps')\quad \text{with}\quad \delta(\eps,\eps')=\sum_{j=1}^{d}\1_{\eps_{j}\ne \eps_{j}'}
\]
\item there exists a constant $a\in [0,1/2]$ such that 
\[
h^{2}\pa{P_{\eps},{P_{\eps'}}}\le \frac{a}{n}\quad \text{for all $\eps,\eps'\in \{-1,1\}^{d}$ satisfying $\delta(\eps,\eps')=1$.}
\]
\end{listi}
Then for all measurable mappings $\widehat P:\sX^{n}\to \cP$,
%
\begin{equation}\label{concAssouad}
\sup_{P\in \cP}\E_{\gP}\cro{h^{2}(P,\widehat P(\bsX))}\ge \frac{d\eta}{8}\max\ac{1-\sqrt{2a},(1-a/n)^{2n}},
\end{equation}
where $\E_{\gP}$ denotes the expectation with respect to a random variable $\bsX=(\etc{X})$ with distribution $\gP=P\on$.
\end{lem}
\begin{proof}
Given a probability $P$ on $(\sX,\cX)$, let $\overline \eps$ be a minimizer over $\{-1,1\}^{d}$ of the mapping $\eps\mapsto h^{2}(P,P_{\eps})$. By definition of $\overline \eps$, for all $\eps\in\{-1,1\}^{d}$
\[
h^{2}(P_{\eps},P_{\overline \eps})\le 2\pa{h^{2}(P,P_{\eps})+h^{2}(P,P_{\overline \eps})}\le 4h^{2}(P,P_{\eps}).
\]
 Hence by~(i), for all $\eps\in\{-1,1\}^{d}$
\[
h^{2}(P_{\eps},P)\ge \frac{\eta}{4}\delta(\eps,\overline \eps)=\sum_{i=1}^{d}\cro{\frac{1+\eps_{i}}{2}\ell_{i}(P)+\frac{1-\eps_{i}}{2}\ell_{i}'(P)}
\]
with $\ell_{i}(P)=(\eta/4)\1_{\overline \eps_{i}=-1}$ and $\ell_{i}'(P)=(\eta/4)\1_{\overline \eps_{i}=+1}$ for  $i\in\{1,\ldots,d\}$. The result follows by applying the version of Assouad's lemma that can be found in Birg\'e~\citeyearpar{MR816706} with $\beta_{i}=a/n$ for all $i\in\{1,\ldots,d\}$, $\alpha=\eta/4$ and the change of notation from $\eps\in\{-1,1\}$ to $\eps\in\{0,1\}$.
\end{proof}

\subsection{Proof of Proposition~\ref{prpL2-h}}
Since the statistical model $\overline \sQ=\{R_{\gamma}=Q_{v^{-1}(\gamma)},\; \gamma\in J\}$ is regular with constant Fisher information equal to 8, by applying Theorem~7.6 page 81 in Ibragimov and Has{'}minski{\u\i}~\citeyearpar{MR620321} we obtain that 
\[
h^{2}\pa{R_{\gamma},R_{\gamma'}}\le \pa{\gamma'-\gamma}^{2}\quad \text{for all $\gamma,\gamma'\in J$}
\]
and for any compact subset $K$ of $J$, there exists a constant $c_{K}>0$
\[
h^{2}\pa{R_{\gamma},R_{\gamma'}}\ge c_{K}^{2}\pa{\gamma'-\gamma}^{2}\quad \text{for all $\gamma,\gamma'\in K$}.
\]
The result follows by substituting $\bsg$ and $\bsg'$ to $\gamma$ and $\gamma'$ respectively and then integrating with respect to $P_{W}$.

\subsection{Proof of Proposition~\ref{prop-bs}}
For $\bsg\in\cH_{\alpha}(M)$ and $j\in\{1,\ldots,D\}$, let $\gamma_{j}=D\int_{I_{j}}\bsg(w)dw$ and $\overline \bsg=\sum_{j=1}^{D}\gamma_{j}\1_{I_{j}}$. Since $\bsg$ takes its values in $J$, $\gamma_{j}\in J$ for all $j\in\{1,\ldots,D\}$ and $\overline \bsg=\sum_{j=1}^{D}\gamma_{j}\1_{I_{j}}\in\overline \bsS$. Since for all $w\in I_{j}$, $|\bsg(w)-\overline \bsg(w)|\le \sup_{|w-w'|\le 1/D}|\bsg(w)-\bsg(w')|\le MD^{-\alpha}$ and $\bsS$ is dense in $\overline \bsS$ with respect to the supremum norm
\begin{align*}
\sup_{\bsg\in\cH_{\alpha}(M)}\inf_{\overline \bsg\in \bsS}\norm{\bsg-\overline\bsg}_{2}&\le \sup_{\bsg\in\cH_{\alpha}(M)}\inf_{\overline \bsg\in \bsS}\norm{\bsg-\overline\bsg}_{\infty}\\
&=\sup_{\bsg\in\cH_{\alpha}(M)}\inf_{\overline \bsg\in \overline \bsS}\norm{\bsg-\overline\bsg}_{\infty}\le MD^{-\alpha}.
\end{align*}

Using~\eref{eq-bsph} and the fact that the data $\etc{X}$ are i.i.d., we deduce that for all functions $\bsg$ and $\bsg'$ with values in $J$, 
\[
\gh^{2}(\gR_{\bsg},\gR_{\bsg'})=nh^{2}(R_{\bsg},R_{\bsg'})\le n{\kappa^{2}}\norm{\bsg-\bsg'}_{2}^{2}\le n{\kappa^{2}}\norm{\bsg-\bsg'}_{\infty}^{2}
\]
and by applying Corollary~\ref{cor-general} with $V=D+1$ we {derive} that 
\begin{align*}
\lefteqn{\sup_{{\bsg\et}\in \cH_{\alpha}(M)}\E\cro{h^{2}(R_{\bsg\et},R_{\widehat \bsg})}}
\hspace{20mm}\\&\le C'\cro{\adjustlimits \sup_{{\bsg\et}\in \cH_{\alpha}(M)}\inf_{\overline \bsg\in \bsS}h^{2}(R_{\bsg\et},R_{\overline \bsg})+\frac{V}{n}\cro{1+\log_{+}(n/V)}}\\
&\le C'\cro{ \kappa^{2}\adjustlimits\sup_{{\bsg\et}\in \cH_{\alpha}(M)}\inf_{\overline \bsg\in \bsS}\norm{\bsg\et-\overline \bsg}_{2}^{2}+\frac{V}{n}\cro{1+\log_{+}(n/V)}}\\
&\le C'\cro{ \kappa^{2} M^{2}D^{-2\alpha}+\frac{D+1}{n}\log(en)}.
\end{align*}
Let us set $L_{n}=\log(en)$. With our choice of $D\ge 1$, 
\[
D-1< \pa{\frac{\kappa^{2} M^{2}n}{L_{n}}}^{\frac{1}{1+2\alpha}}\le D
\]
hence $ \kappa^{2} M^{2}D^{-2\alpha}\le DL_{n}/n$, $D<1+(\kappa^{2} M^{2}n/L_{n})^{\frac{1}{1+2\alpha}}$ and the result follows from the inequalities
\begin{align*}
&\kappa^{2} M^{2}D^{-2\alpha}+\frac{(D+1)L_{n}}{n}\le 2\frac{DL_{n}}{n}+\frac{L_{n}}{n}\le 2\cro{\frac{(\kappa M)^{1/\alpha}L_{n}}{n}}^{\frac{2\alpha}{1+2\alpha}}+\frac{3L_{n}}{n}.
\end{align*}

\subsection{Proof of Proposition~\ref{prop-bi}}
Let ${a_{0}}$ be the middle of the interval $K$ of length $2\overline L$. Given $N\ge 1$, $L>0$ and $\eps\in\{-1,1\}^{2^{N}}$, we define $\bsg_{\eps}={a_{0}}+G_{\eps}$ with $G_{\eps}$ as  in Proposition~\ref{propert}. Provided that $L\le \overline L\wedge L_{0}$ with $L_{0}=2^{-[(N-1)\alpha+1]}M$, the functions $\bsg_{\eps}$ takes their values in $K\subset J$ and satisfies~\eref{def-Hold} and consequently belongs to $\cH_{\alpha}(M)$ for all $\eps\in\{-1,1\}^{2^{N}}$. Let $R_{\eps}=R_{\bsg_{\eps}}$ for all $\eps\in \{-1,1\}^{2^{N}}$ and, as in the proof of Proposition~\ref{prop-BI0}, set $P_{\bsg}=R_{\bsg}\cdot P_{W}$ and ${P_{\varepsilon}}=P_{\bsg_{\eps}}$ for short. Integrating the inequalities~\eref{eq-bsph} and~\eref{eq-bifh} with respect to $P_{W}$ and using that  for all $\eps,\eps'\in\{-1,1\}^{2^{N}}$, $\norm{G_{\eps}-G_{\eps'}}_{2}=\norm{\bsg_{\eps}-\bsg_{\eps'}}_{2}$ we obtain that
\[
c_{K}^{2}\norm{G_{\eps}-G_{\eps'}}_{2}^{2}\le h^{2}\pa{R_{\eps},R_{\eps'}}\le \kappa^{2}\norm{G_{\eps}-G_{\eps'}}_{2}^{2}.
\]
Since $P_{W}$ is the uniform distribution on $[0,1]$, by arguing as in the proof of Proposition~\ref{prop-BI0}
\[
\norm{G_{\eps}-G_{\eps'}}_{2}^{2}=4L^{2}2^{-N}\|g\|_{2}^{2}\delta(\eps,\eps')\quad \text{for all $\eps,\eps'\in\{-1,1\}^{2^{N}}$}
\]
and consequently, provided that $L$ satisfies 
\begin{equation}\label{eq-condLc2}
L\le \overline L\wedge L_{0}\wedge \pa{4\kappa \|g\|_{2}\sqrt{2^{-(N-1)}n}}^{-1}
\end{equation}
the family of probabilities $\sC=\{P_{\eps},\; \eps\in\{-1,1\}^{|\Lambda|}\}$ is a subset of $\cP=\{P_{\bsg},\; \bsg\in\cH_{\alpha}(M)\}$ that fulfils the assumptions of Lemma~\ref{assouad} with $d=2^{N}$, 
\[
\eta=4c_{K}^{2}L^{2}2^{-N}\|g\|_{2}^{2}\quad \text{and}\quad a=4n\kappa^{2}L^{2}2^{-N}\|g\|_{2}^{2}\le 1/8.
\]
We derive from~\eref{concAssouad} that
\begin{equation}\label{eq-LB00b}
\cR_{n}(\cH_{\alpha}(M))\ge \frac{c_{K}^{2}\|g\|_{2}^{2}L^{2}}{2}\pa{1-\sqrt{2a}}\ge  \frac{c_{K}^{2}\|g\|_{2}^{2}L^{2}}{4}.
\end{equation}
If $\kappa^{2}\norm{g}_{2}^{2}M^{2}n> 1/8$, we choose  $N\ge 2$ such that 
\[
2^{N}\ge \pa{2^{{2(2+\alpha)}}\kappa^{2}\norm{g}_{2}^{2}M^{2}n}^{1/(1+2\alpha)}>2^{N-1}
\]
and $N=1$ otherwise. In any case, our choice of $N$ satisfies
\[
L_{0}=2^{-[(N-1)\alpha+1]}M\le  \pa{4\kappa \|g\|_{2}\sqrt{2^{-(N-1)}n}}^{-1}.
\]
When $N\ge 2$, 
\begin{align*}
L_{0}^{2}&=2^{-2\alpha(N-1)-2}M^{2}\ge \frac{M^{2}}{4} \pa{2^{{2(2+\alpha)}}\kappa^{2}\norm{g}_{2}^{2}M^{2}n}^{-\frac{2\alpha}{1+2\alpha}}\\
&=\pa{\frac{M^{1/\alpha}}{2^{2\alpha+6+1/\alpha}\kappa^{2}\norm{g}_{2}^{2}n}}^{\frac{2\alpha}{1+2\alpha}}=L_{1}^{2},
\end{align*}
while $L_{0}=M/2$ when $N=1$. The choice $L=\overline L\wedge L_{1}\wedge (M/2)$ satisfies~\eref{eq-condLc2} and we deduce from the equalities 
\[
h^{2}(R_{\eps},R_{\eps'})=\int_{\sW}h^{2}\pa{R_{\bsg_{\eps}(w)},R_{\bsg_{\eps'}(w)}}dP_{W}(w)=h^{2}(P_{\eps},P_{\eps'})
\]
and~\eref{eq-LB00b} that 
\begin{align*}
\cR_{n}(\cH_{\alpha}(M))& \ge \frac{c_{K}^{2}\|g\|_{2}^{2}}{4}\cro{\pa{\frac{M^{1/\alpha}}{2^{2\alpha+6+1/\alpha}\kappa^{2}\norm{g}_{2}^{2}n}}^{\frac{2\alpha}{1+2\alpha}}\bigwedge \pa{\frac{M^{2}}{4}}\bigwedge \overline L^{2}}.
\end{align*}
The conclusion follows by {choosing} $g(x)=x\1_{[0,1/2]}+(1-x)\1_{[1/2,1]}$ which satisfies $\|g\|_{2}^{2}=1/12$.

\subsection{Proof of Proposition~\ref{prop-BI0}}
When the Poisson family is parametrized by the mean, given two functions $\bsg,\bsg'$ mapping $\sW=[0,1]$ into $J=(0,+\infty)$, The Hellinger-type distance $h^{2}(R_{\bsg},R_{\bsg'})$ can be written as 
\begin{equation}\label{h-L2-Poi}
h^{2}(R_{\bsg},R_{\bsg'})=\int_{\sW}\cro{1-e^{-\pa{\sqrt{\bsg(w)}-\sqrt{\bsg'(w)}}^{2}/2}}dP_{W}(w).
\end{equation}
Using that for all $x\in [0,1]$, $(1-e^{-1})x\le 1-e^{-x}\le x$, we deduce from~\eref{h-L2-Poi} that
\begin{equation}\label{eq-hL2}
{ \frac{1}{2}}(1-e^{-1})\norm{\sqrt{\bsg}-\sqrt{\bsg'}}_{2}^{2}\le h^{2}(R_{\bsg},R_{\bsg'})\le \frac{1}{2}\norm{\sqrt{\bsg}-\sqrt{\bsg'}}_{2}^{2}
\end{equation}
whenever $\norm{\sqrt{\bsg}-\sqrt{\bsg'}}_{\infty}\le 1$.  

Let $N$ be some positive integer, $L$ some positive number and $g$ a $1$-Lipschitz function supported on $[0,1]$ with values in $[-b,b]$. Let us set $\Lambda=\{0,\ldots,2^{N}-1\}$ and for $\eps\in\{-1,1\}^{|\Lambda|}$, $G_{\eps}$ the function defined by~\eref{def-G} and $\bsg_{\eps}=L+G_{\eps}$. Under our assumption on $g$, $\bsg_{\eps}$ takes its values in $[(1-b)L,(1+b)L]$ and by Proposition~\ref{propert}, $\bsg_{\eps}$ satisfies~\eref{def-Hold} provided that $L\le 2^{-[(N-1)\alpha+1]}M$. Hence, under the conditions $L\le 2^{-[(N-1)\alpha+1]}M$ and $b<1$, $\bsg_{\eps}$ belongs to $\cH_{\alpha}(M)$ for all  $\eps\in\{-1,1\}^{|\Lambda|}$. For all $\eps,\eps'\in\{-1,1\}^{|\Lambda|}$, 
\[
\frac{\ab{G_{\eps}-G_{\eps'}}}{2\sqrt{(1+b)L}}\le \ab{\sqrt{\bsg_{\eps}}-\sqrt{\bsg_{\eps'}}}=\frac{\ab{\bsg_{\eps}-\bsg_{\eps'}}}{\sqrt{\bsg_{\eps}}+\sqrt{\bsg_{\eps'}}}\le \frac{\ab{G_{\eps}-G_{\eps'}}}{2\sqrt{(1-b)L}},
\]
and %
\[
\ab{\sqrt{\bsg_{\eps}}-\sqrt{\bsg_{\eps'}}}\le \sqrt{(1+b)L}-\sqrt{(1-b)L}=\cro{\sqrt{1+b}-\sqrt{1-b}}\sqrt{L}.
\]
In particular, $\norm{\sqrt{\bsg_{\eps}}-\sqrt{\bsg_{\eps'}}}_{\infty}\le 1$ for 
\[
L\le \pa{\sqrt{1+b}-\sqrt{1-b}}^{-2}=\frac{1+\sqrt{1-b^{2}}}{2b^{2}}=L_{0}
\]
and, writing $R_{\eps}$ for $R_{\bsg_{\eps}}$ for short, it follows from~\eref{eq-hL2} that 
\begin{equation}\label{eq-hL20}
{\frac{(1-e^{-1})}{8(1+b)L}}\norm{G_{\eps}-G_{\eps'}}_{2}^{2}\le  h^{2}(R_{\eps},R_{\eps'})\le \frac{1}{8(1-b)L}\norm{G_{\eps}-G_{\eps'}}_{2}^{2}.
\end{equation}
Since $P_{W}$ is the uniform distribution and the supports of the functions $g_{k}:x\mapsto g(2^{N}x-k)$ for $k\in \Lambda$ are disjoint, we obtain that for all $\eps,\eps'\in\{-1,1\}^{|\Lambda|}$ 
\begin{align*}
\norm{G_{\eps}-G_{\eps'}}_{2}^{2}&=L^{2}\sum_{k\in \Lambda}\int_{I_{k}}\pa{\eps_{k+1}g_{k}(x)-\eps_{k+1}'g_{k}(x)}^{2}dx\\
&=L^{2}\sum_{k\in \Lambda}{\ab{\eps_{k+1}-\eps_{k+1}'}^{2}}\int_{I_{k}}g_{k}^{2}(x)dx={4}L^{2}2^{-N}\norm{g}_{2}^{2}\delta(\eps,\eps').
\end{align*}
Let us denote by $P_{\bsg}=R_{\bsg}\cdot P_{W}$ the probability associated to $R_{\bsg}$ and write $P_{\eps}$ for $P_{\bsg_{\eps}}$ for short.
We deduce from~\eref{eq-hL20} that provided that $L$ and $b$ satisfy
\begin{equation}\label{cond-L}
L\le \pa{2^{-[(N-1)\alpha+1]}M}\bigwedge \frac{1+\sqrt{1-b^{2}}}{2b^{2}}\bigwedge {{\frac{(1-b)2^{N-3}}{\norm{g}_{2}^{2}n}}},
\end{equation}
the family of probabilities $\sC=\{P_{\eps},\; \eps\in\{-1,1\}^{|\Lambda|}\}$ is a subset of $\{P_{\bsg},\; \bsg\in\cH_{\alpha}(M)\}$ that fulfils the assumptions of Assouad's lemma (Lemma~\ref{assouad}) with $d=|\Lambda|=2^{N}$, 
\[
\eta=\frac{(1-e^{-1})L2^{-(N+1)}\norm{g}_{2}^{2}}{1+b}\quad \text{and}\quad a={\frac{nL2^{-N}\norm{g}_{2}^{2}}{1-b}}\in [0,1/8].
\]
We derive from the equalities
\[
h^{2}(R_{\eps},R_{\eps'})=\int_{\sW}h^{2}\pa{R_{\bsg_{\eps}(w)},R_{\bsg_{\eps'}(w)}}dP_{W}(w)=h^{2}(P_{\eps},P_{\eps'})
\]
and~\eref{concAssouad} that
\begin{equation}\label{eq-LB00}
\cR_{n}(\cH_{\alpha}(M))\ge \frac{(1-e^{-1})\norm{g}_{2}^{2}L}{16(1+b)}\pa{1-\sqrt{2a}}\ge \frac{(1-e^{-1})\norm{g}_{2}^{2}L}{32(1+b)}.
\end{equation}
If $\norm{g}_{2}^{2}Mn> (1-b)/2$, we choose $N\ge 2$ such that
\begin{align*}
2^{N}\ge \cro{\frac{2^{2+\alpha}\norm{g}_{2}^{2}Mn}{1-b}}^{\frac{1}{1+\alpha}}>2^{N-1}.
\end{align*}
Otherwise, we choose $N=1$. Note that in any case, 
\[
2^{-[(N-1)\alpha+1]}M\le {\frac{(1-b)2^{N-3}}{n\norm{g}_{2}^{2}}}.
\]
Besides, if $N\ge 2$ 
\begin{align*}
2^{-[(N-1)\alpha+1]}M&=2^{-1}M 2^{-(N-1)\alpha}\ge 2^{-1}M \cro{\frac{2^{2+\alpha}\norm{g}_{2}^{2}Mn}{1-b}}^{-\frac{\alpha}{1+\alpha}}\\
&=\pa{\frac{(1-b)M^{\frac{1}{\alpha}}}{2^{3+\alpha+1/\alpha}\norm{g}_{2}^{2}n}}^{\frac{\alpha}{1+\alpha}}=L_{1}
\end{align*}
while for $2^{-[(N-1)\alpha+1]}M=M/2$ for $N=1$. Finally, we choose $L=L_{0}\wedge L_{1}\wedge (M/2)$, which satisfies~\eref{cond-L}, and we derive from ~\eref{eq-LB00} that  
\begin{align*}
 \lefteqn{\cR_{n}(\cH_{\alpha}(M))}\hspace{10mm}\\
 &\ge \frac{(1-e^{-1})\|g\|_{2}^{2}}{32(1+b)}\cro{\pa{\frac{(1-b)M^{\frac{1}{\alpha}}}{2^{3+\alpha+1/\alpha}\norm{g}_{2}^{2}n}}^{\frac{\alpha}{1+\alpha}}\wedge \frac{M}{2} \wedge \frac{1+\sqrt{1-b^{2}}}{b^{2}}}.
\end{align*}
The conclusion follows by taking $g(x)=x\1_{[0,1/2]}+(1-x)\1_{[1/2,1]}$ for which $b=1/2$ and $\|g\|_{2}^{2}=1/12$ .

\subsection{Proof of Proposition~\ref{prop-rhoEstPoisson}}
Let $\alpha\in (0,1]$ and $M>0$. For a function $\bsg\in\cH_{\alpha}(M)$ taking values in $J=(0,+\infty)$, we set $\gamma_{j}=D\int_{I_{j}}\bsg(w)dw$, for $j\in\{1,\ldots,D\}$ and $\overline\bsg=\sum_{j=1}^{D}\gamma_{j}\1_{I_{j}}$. As an immediate consequence, $\gamma_{j}\in J$ for all $j\in\{1,\ldots,D\}$ and $\overline\bsg\in\overline \bsS$. Since for all $w\in I_{j}$, with the fact that $\bsg\in\cH_{\alpha}(M)$ $$\left|\bsg(w)-\overline\bsg(w)\right|\le \sup_{|w-w'|\le 1/D}\left|\bsg(w)-\bsg(w')\right|\le MD^{-\alpha}$$ and $\bsS$ is dense in $\overline \bsS$ with respect to the supremum norm, we derive
\begin{align*}
\sup_{\bsg\in\cH_{\alpha}(M)}\inf_{\overline\bsg\in\bsS}\norm{\sqrt{\bsg}-\sqrt{\overline\bsg}}_{2}&\le \sup_{\bsg\in\cH_{\alpha}(M)}\inf_{\overline\bsg\in\bsS}\norm{\sqrt{\bsg}-\sqrt{\overline\bsg}}_{\infty}\\
&\le\sup_{\bsg\in\cH_{\alpha}(M)}\inf_{\overline\bsg\in\bsS}\sqrt{\norm{\bsg-\overline\bsg}_{\infty}}\\
&=\sup_{\bsg\in\cH_{\alpha}(M)}\inf_{\overline\bsg\in\overline\bsS}\sqrt{\norm{\bsg-\overline\bsg}_{\infty}}\\
&\le\sqrt{M}D^{-\frac{\alpha}{2}}.
\end{align*}

Using the fact that the data $\etc{X}$ are i.i.d. and $1-e^{-x}\le x$ for all $x\in[0,+\infty)$, we deduce that for all functions $\bsg$ and $\bsg'$ with values in $J=(0,+\infty)$,
\begin{align}
\gh^{2}(\gR_{\bsg},\gR_{\bsg'})=nh^{2}(R_{\bsg},R_{\bsg'})&=n\int_{\sW}\cro{1-e^{-\pa{\sqrt{\bsg(w)}-\sqrt{\bsg'(w)}}^{2}/2}}dP_{W}(w)\nonumber\\
&\le\frac{n}{2}\norm{\sqrt{\bsg}-\sqrt{\bsg'}}_{2}^{2}.\label{Poi-Con}
\end{align}
Applying Corollary~\ref{cor-general} with $V=D+1$ together with \eref{Poi-Con}, we obtain that 
\begin{align*}
\sup_{{\bsg\et}\in \cH_{\alpha}(M)}\E\cro{h^{2}(R_{\bsg\et},R_{\widehat \bsg})}&\le C'\cro{\adjustlimits \sup_{{\bsg\et}\in \cH_{\alpha}(M)}\inf_{\overline \bsg\in \bsS}h^{2}(R_{\bsg\et},R_{\overline \bsg})+\frac{V}{n}\cro{1+\log_{+}(n/V)}}\\
&\le C'\cro{\frac{1}{2}\adjustlimits\sup_{{\bsg\et}\in \cH_{\alpha}(M)}\inf_{\overline \bsg\in \bsS}\norm{\sqrt{\bsg\et}-\sqrt{\overline\bsg}}_{2}^{2}+\frac{V}{n}\cro{1+\log_{+}(n/V)}}\\
&\le C'\cro{\frac{1}{2}MD^{-\alpha}+\frac{D+1}{n}\log(en)}.
\end{align*}
Let us set $L_{n}=\log(en)$. With our choice of $D\ge 1$, 
\[
D-1< \pa{\frac{Mn}{2L_{n}}}^{\frac{1}{1+\alpha}}\le D
\]
hence $MD^{-\alpha}/2\le DL_{n}/n$, $D<1+\cro{Mn/(2L_{n})}^{1/(1+\alpha)}$ and the result follows from the inequalities
\begin{align*}
&\frac{MD^{-\alpha}}{2}+\frac{(D+1)L_{n}}{n}\le 2\frac{DL_{n}}{n}+\frac{L_{n}}{n}\le 2\cro{\frac{(M/2)^{1/\alpha}L_{n}}{n}}^{\frac{\alpha}{1+\alpha}}+\frac{3L_{n}}{n}.
\end{align*}

\bibliographystyle{apalike}

\end{document}